\documentclass{psp}

\ifprodtf \else
  \checkfont{msam10}
  \iffontfound
    \IfFileExists{amsfonts.sty}
      {\typeout{^^JFound AMS Symbol fonts on the system, using the
                'amsfonts' package.^^J}%
       \usepackage{amsfonts}%

      }
      {\providecommand\mathbb[1]{\mathsf{##1}}
       \providecommand\mathfrak[1]{\mathcal{##1}}}
  \else
    \providecommand\mathbb[1]{\mathsf{#1}}
    \providecommand\mathfrak[1]{\mathcal{#1}}
  \fi

\fi

\ifprodtf \else
  \IfFileExists{amsbsy.sty}
    {\typeout{^^JFound the 'amsbsy' package on the system, using it.^^J}%
     \usepackage{amsbsy}}
    {}

\fi

\usepackage{amssymb}
\usepackage{amsmath}
\usepackage{epsfig}

\usepackage[latin1]{inputenc} 
\usepackage[T1]{fontenc}

\newtheorem{The}{Theorem}[section]
\newtheorem{Lem}[The]{Lemma}
\newtheorem{Cor}[The]{Corollary}
\newtheorem{Pro}[The]{Proposition}

\newnumbered{Hyp}[The]{Hypothesis}
\newnumbered{definition}[The]{Definition}
\newnumbered{Not}[The]{Notation}
\newnumbered{Assertion}[The]{Assertion}

\newnumbered{Exam}[The]{Example}

\newnumbered{remark}[The]{Remark}
\newnumbered{remarks}[The]{Remarks}

\def\a{\alpha}

\def\A{\mathbf A}

\def\b{\beta}
\def\C{\mathbf C}
\def\D{\Delta}
\def\d{\delta}

\def\f{\phi}

\newromanexpr\F{F}

\def\g{\gamma}

\def\J{{\mathcal{J}}}
\def\l{\lambda}
\def\L{\mathbf{L}}

\newromanexpr\ord{ord}

\newromanexpr\orb{orb}

\newromanexpr\geom{geom}

\def\p{\pi}

 \def\P{{\mathbf P}}

\def\Q{\mathbf Q}
\def\R{\mathbf R}
\def\r{\rho}

\def\s{\sigma}
\def\t{\tau}

\def\w{\omega}

\def\y{\wedge}

\def\Z{\mathbf Z}
\def\z{\zeta}

\def\limproj{\mathop{\oalign{lim\cr\hidewidth$\longleftarrow$\hidewidth\cr}}}

\title[Geometric motivic Poincar{\'e} series of q.o.~singularities]
  {Geometric motivic Poincar{\'e} series of quasi-ordinary singularities}

\author[ H. Cobo Pablos and P.~D.~Gonz\'alez P\'erez ]
         {HELENA COBO PABLOS \thanks{Supported by Fundaci\'on Caja Madrid and MTM2004-08080-C02-01 grant
          of MEC Spain}\\
         Department of Mathematics,
     University of Leuven, \addressbreak
     Celestijnenlaan 200B,
B-3001 Leuven-Heverlee, Belgium
          \and\ PEDRO D. GONZ\'ALEZ P\'EREZ  \thanks{Supported by {\em Programa Ram\'on y Cajal} and MTM2004-08080-C02-01 grant
          of MEC Spain} \\ ICMAT. Depto. \'Algebra.
          Universidad Complutense de Madrid.  \addressbreak
Plaza de las Ciencias 3. 28040. Madrid. Spain
          }

\begin{document}

\maketitle

 \begin{abstract}
The
{\em geometric motivic Poincar\'e series} of a germ  $(S,0)$ of
complex algebraic variety takes into account the classes in
the Grothendieck ring of the jets of arcs through $(S,0)$. Denef and Loeser
proved that this series has a rational form. We give an explicit description of this invariant when
$(S,0)$ is an irreducible germ of {\em quasi-ordinary hypersurface
singularity}
in terms of  the Newton polyhedra of the  {\em
logarithmic jacobian ideals}. These ideals are determined by the {\em characteristic monomials}
of a quasi-ordinary branch parametrizing $(S,0)$.
\end{abstract}

\section*{Introduction}
A germ $(S,0)$ of complex analytic variety equidimensional of
dimension $d$ is {\em quasi-ordinary} (q.o.) if there exists a
finite map $\pi: (S,0)\rightarrow(\C^d,0)$ which is unramified
outside a normal crossing divisor in $(\C^d,0)$. Quasi-ordinary
singularities admit fractional power series parametrizations,
which generalize Newton-Puiseux expansions of plane curves (see
\cite{Abhyankar}). Quasi-ordinary surface singularities appear
classically in the Jung's method to parametrize and resolve
surface singularities (see \cite{Jung}) and are related to the
 classification of singularities by Zariski's
dimensionality type (see \cite{Lipman-Eq}). Classical examples are
plane curve singularities, Hirzebruch-Jung surfaces and simplicial
toric varieties. In addition to the applications in
equisingularity problems, the
 class of q.o.~singularities is also of interest
to test and study various open questions and conjectures for
singularities in general, particularly in the hypersurface case
(see \cite{Nemethi-Pregunta}). In many cases the results passed by
using the fractional power series parametrizations of these
singularities, which allow  explicit computations combining
analytic, topological and combinatorial arguments (see for
instance \cite{Lipman2, Gau, PPP04, Nemethi, ACLM, P-F}). It is
natural to investigate new invariants of singularities, such as
those arising in the development of motivic integration, on this
class of singularities with the hope to extend the methods or
results to wider classes (for instance by passing through Jung's
approach).

We recall the definition of the {\em geometric  motivic Poincar\'e
series} of a germ
 $(Z,0)$ of complex algebraic variety (or
complex analytic or algebroid) equidimensional of dimension $d$.
The set $H_Z$ of formal arcs of the form, $\Spec \C[[t]]
\rightarrow (Z,0)$ has a scheme structure over $\C$
(not necessarily of finite type). Let $s \geqslant 0$, the set
$H_{s,Z}$ of $s$-jets of $(Z, 0)$ of the form  $ \Spec
\C[t] / (t^{s+1}) \rightarrow (Z,0)$, has the structure of
algebraic variety over $\C$. By a Theorem of Greenberg
\cite{Greenberg}, the image of $H_Z$ by the natural morphism of
schemes $j_s : H_Z \rightarrow H_{s, Z}$, which maps an arc to its
$s$-jet,
  is a constructible subset          of     $ H_{s,Z}$.

The {\em Grothendieck ring} $ K_0 (\Var_\C) $
of $\C$-{\em varieties} is generated by the symbols $[X]$ for $X$
an algebraic variety,  subject to relations: $[X] = [X']$ if $X$
is isomorphic to $X'$, $[X] = [X- X'] + [X'] $ if $X'$ is closed
in $X$ and $[X][X'] = [X \times X']$.
 Since a constructible set $W$  has an image $[W]$ in the {\em
Grothendieck ring}  $K_0 (\Var_\C) $ of varieties,
it is natural to consider the {\em geometric motivic Poincar\'e
series} of the germ $(Z,0)$:
\[
P^{(Z,0)}_{\geom}  (T) := \sum_{s \geqslant 0} [j_s
(H_Z)] T^s \in K_0 (\Var_\C) [[T]].
\]
 This series, which was introduced more generally
by Denef and Loeser in \cite{DL1}, is inspired
by related Poincar\'e series in arithmetic geometry \cite{DL-R}. It
has a rational form:
\begin{The*}   (see \cite{DL1}   Theorem 1.1) Set $\L:= [\A_{\C}^1]$.
      There exist $a_i \in
\Z$ and $b_i \in \Z_{\geqslant 1}$,  for $i=1, \dots,r$ and $Q(T)
\in K_0(\Var_\C) [\L ^{-1}]  [T]$ such that
\[
 P_{\geom}^{(Z,0)} (T)= Q(T) \prod_{i=1}^r (1 -
\L^{a_i} T^{b_i})^{-1} \mbox{ in } K_0 (\Var_\C) [\L ^{-1}]  [[T]].
\]
\end{The*}

 There is not a general formula
for this invariant in terms of a resolution of singularities of
$(Z,0)$ (notice that these kind of formulas exists for other
motivic invariants as the {\em motivic zeta function}, see
\cite{DL-bcn}). There is no conjecture on the meaning of the
exponents $a_i$, $b_i$ which may appear in the denominator of a rational form of $P_{\geom}^{(Z,0)} (T)$. It is not known if such a rational expression for this series  holds in the ring $K_0
(\Var_\C)  (T)$.

If $(Z, 0)$ is an analytically irreducible germ of plane curve the
series  $P^{(Z,0)}_{\geom}  (T)$ determines and it is
determined  by  the multiplicity of $(Z,0)$ (see \cite{DL2}).
   Nicaise has given formulas for $P^{(Z,0)}_{\geom} (T)
$ for germs $(Z,0)$ with a {\em very special} embedded resolution
of singularities \cite{Nicaise2}, for instance if $(Z,0)$ is the
cone over a smooth hypersurface $H \subset \P^d_\C$ then
$P^{(Z,0)}_{\geom} (T)$ is determined by $[H]$,  $d$ and
the multiplicity of $(Z,0)$.
 Lejeune-Jalabert and Reguera have
studied the motivic Poincar\'e series of a germ $(Z,0)$ of normal
toric surface at its distinguished point. They have given a
formula for the rational form of this series in terms of the
Hirzebruch-Jung continued fraction describing the resolution of
singularities of $Z$ (see \cite{LR} and also \cite{Nicaise} for a
different approach and comparison with other motivic series). If
$(Z,0)$ is a germ of affine toric variety of dimension $d$ we
prove in \cite{CoGP} that $P^{(Z,0)}_{\geom} (T)$ is
determined by the Newton polyhedra of the {\em logarithmic
jacobian ideals}, which are defined in terms of the modules of
differential forms with logarithmic poles outside the torus of
$Z$.  Rond studied the coefficients  of the series $P^{(S,0)}_{geom}
(T)$ associated to  a germ of q.o.~hypersurface and computed the
sum of this series in some particular cases (see \cite{Rond}).

In this paper we describe the rational form of the geometric motivic
Poincar\'e series of a germ $(S, 0)$ of q.o.~hypersurface
singularity. Our approach is independent of Rond's.

A q.o.~hypersurface singularity  $(S, 0)$ has a fractional power series
parametrization, which possesses a finite set of {\em characteristic
monomials} (generalizing the characteristic exponents of plane
branches) and which classify the embedded topological class of
$(S, 0) \subset (\C^{d+1},0)$ (see \cite{Gau, Lipman2}).
Since the normalization $(\bar{S}, 0)$ of the germ $(S,0)$ is a
toric singularity   (see
\cite{GP4}) it is natural to extend the approach in \cite{CoGP} to
the q.o.~case. For this reason we introduce a sequence of
monomial ideals $\J_1, \dots, \J_d$ of the analytic algebra of the normalization
$(\bar{S},0)$. These ideals, called the {\em logarithmic jacobian
ideals},  are defined first in a combinatorial manner in terms of
the characteristic monomials. Inspired by the toric
case \cite{CoGP, LR}, we prove that these ideals can be defined in
terms of the composite of canonical maps $\Omega_S \rightarrow
\Omega_{\bar{S}} \rightarrow \Omega_{\bar{S}} (\log D)$, where
$\Omega_S$ (resp. $\Omega_{\bar{S}}$ and $\Omega_{\bar{S}}(\log
D)$) denotes the module of holomorphic differential forms
of
 $(S,0)$ (resp. over $(\bar{S},0)$ and the module of forms with
logarithmic poles at the complement $D$ of the torus in
$({\bar{S}},0)$), see Section \ref{torQO}. In the toric case the
blow up of the ideal $\J_d$ is the Nash modification (see
\cite{LR}). In the  q.o.~case the normalized Nash modification of
$(S,0)$ is equal to the normalization of $(S,0)$ followed by
the normalized blow-up of the ideal $\J_d$ (see \cite{PedroNash}).

We study the arc space $H_{S}$ of $(S,0)$ by relating it with the
arc space $H_{\bar{S}}$ of the normalization $(\bar{S},0)$. We
denote by $H_{\bar{S}}^*$   (resp. by $H_{S}^*$)  the set of arcs
with generic point in the torus of $\bar{S}$ (resp. arcs in $H_S$
which lift to an arc in $H^*_{\bar{S}}$). For $s \geqslant 0$ we
have that $j_s (H_{\bar{S}}) = j_s (H_{\bar{S}}^*)$ (see
\cite{Nicaise}) but $j_s (H_{{S}}) \ne  j_s (H_{{S}}^*)$ in
general if $S$ is not normal (the simplest example is Whitney
umbrella). To avoid this difficulty we
consider a finite set of q.o.~coordinate sections $\{ (S_\theta,
0) \}_{\theta} $ of $(S,0)$, with $S_\theta  = S$ if $\theta =0$,
which are compatible with the toric structure of the
normalization. We define an auxiliary motivic series $P(S)$ whose
coefficients are the image in the Grothendieck ring of the
$s$-jets of arcs in $H_S^*$  which are not jets of arcs through
any proper q.o.~coordinate section $(S_\theta, 0)$, ${\theta \ne
0} $. It follows from this that  $
    P_{
{\geom}}^{(S,0)}(T) = \sum_{\theta} P(S_\theta)$. This reduces the
study of    $P_{\geom}^{(S,0)}(T)$ to the study of
$P(S)$.

Since $(\bar{S},0)$ is a toric singularity, the arc space of the
torus of $\bar{S}$ acts on $H_{\bar{S}}$. We use the orbit
decomposition of  $H_{\bar{S}}$ under this action studied by Ishii
in \cite{Ishii-algebra, Ishii-crelle, Ishii-fourier}.
 The set $H^*_S$
decomposes as a disjoint union of orbits $H^*_{S,\nu}$,
parametrized by the arc space of the torus,  where $\nu$ runs
through certain subset of the lattice $N$ of one-parametric
subgroups of the torus. We prove that the $s$-jets of these orbits
are locally closed subsets which are either disjoint or equal and
we characterize the equality in combinatorial terms. Then we prove that
the coefficient of $T^m$ in the series $P(S)$ expand as a sum of
classes $[j_m (H_{S, \nu}^*)]$, where $\nu$ runs through a finite
subset of the lattice $N$.

The description of the  rational form of the
geometric motivic Poincar\'e series $P_{
{\geom}}^{(S,0)}(T)$, is done by using the methods and combinatorial results of \cite{CoGP}.

 We determine first a formula for the
class $[j_m (H_{S, \nu}^*)]$ in terms of the Newton polyhedra of the logarithmic jacobian ideals in Theorem
\ref{keyQO}. These ideals satisfy similar combinatorial properties in the
toric case and in the q.o.~case, though the proofs are more
difficult in the second case. In addition, new
combinatorial features are needed to relate the parametrization of
$(S,0)$ with fractional power series with the parametrization of
$j_m (H^*_{S, \nu})$ by using the coordinates of the arc space of
the torus.

We prove then that the rational form of the series
 $P(S)$ is determined by the Newton polyhedra of the
logarithmic jacobian ideals and the lattice $N$ (see Theorem
\ref{PLambdaRacQO}).  We deduce a formula for the rational form of  $P_{
{\geom}}^{(S,0)}(T)$, which holds in
the ring  $\Z[\L](T)\subset K_0 ({
  \Var}_\C) (T)$, see Corollary \ref{P-geomQO}.
In particular a finite set of \textit{candidate poles} of $P_{\geom}^{(S,0)}(T)$ is obtained.
Our result implies that the embedded topological type of the germ
$(S,0) \subset (\C^{d+1}, 0)$ determines the series $P_{
{\geom}}^{(S,0)}(T)$. The converse is not true even if $d=1$. As an
application we deduce a formula for the {\em motivic volume} of
the arc space of a q.o.~hypersurface germ $(S,0)$ in terms of the
logarithmic jacobian ideal $\mathcal J_d$ (see Corollary
\ref{vol-mot2QO}).

 Notice that the series $P_{
{\geom}}^{(S,0)}(T)$ and $P_{\geom}^{(\bar{S},0)}(T)$ may
be quite different (for instance if the normalization  of $(S,0)$
is smooth). In Section \ref{QOexample} we give an example of q.o.~
surface $(S,0)$ such that $P_{\geom}^{(S,0)}(T) \ne
P_{\geom}^{(Z(S),0)}(T)$, where $Z(S)$ is the {\em
monomial variety} associated to the q.o.~hypersurface (see
Section \ref{secQO}). Notice that these two series coincide in the
one dimensional case.

The paper is organized as follows. The first three sections
introduce the basic notions we need on arc and jets spaces, toric
geometry and  q.o.~singularities. We describe the orbit
decomposition of the arc space of a q.o.~hypersurface in Section
\ref{ArcsAndJetsQO}. In Section \ref{mainQO} we state the main
results.
 In Section \ref{conv-newtonQO} we
give some combinatorial convexity properties of the Newton
polyhedra of the logarithmic jacobian ideals. In Section
\ref{tor-jetQO}  we prove Theorem \ref{keyQO}.
 In Section \ref{GenptOrbitsQO} we describe the series $P(S)$ and prove the
 rationality results.
 In Section \ref{torQO} we define the logarithmic jacobian
ideals in terms of differential forms.

The results and proofs in this paper hold if the field $\C$ is replaced
by any algebraically closed field of characteristic
zero.
\section{Basic definitions on arc and jet spaces}
\label{intro-arcsQO}

 We refer to \cite{Ishii-07, EinMustata,DL-bcn,Loo,
Veys} for expository papers on arc and  jet schemes and/or motivic
integration. We introduce arc and jet spaces on an equidimensional
germ of complex analytic variety $(Z,0)$ (or complex algebraic
variety or algebroid). Arc  and jet spaces can be defined on any
algebraic variety  (without fixing the origin of the arcs).

We have an embedding $(Z,0) \subset (\A^n_\C, 0)$  in such a way
that the germ   $(Z,0) $ is defined by the ideal  $I \subset \C \{
x_1, \dots, x_n \} $. An arc $h: \Spec  \C[[t]] \rightarrow
(\A^n_\C, 0)$ (resp. a $m$-jet $\Spec \C [t] /(t^{m+1})
\rightarrow (Z, 0)$)
 is defined by $n$ formal power series
\begin{equation} \label{arc-expQO}
x_i (t) =  a_i^{(1)} t + a_i^{(2)}t^2 + \cdots + a_i^{(r)}t^r +
\cdots      , \quad  i=1,\dots, n,
\end{equation}
(resp. $n$-polynomial expressions (\ref{arc-expQO}) $\mod
t^{m+1}$). For any $F \in I$, the coefficient of $t^k$ in the
series $ F(x_1(t), \dots, x_n (t))$ is a polynomial expression
$\a_F^{(k)} ( \underline{a}^{(1)} , \dots, \underline{a}^{(k)})$,
where $\underline{a}^{(j)} := ( a^{(j)}_1, \dots, a^{(j)}_n )$,
for $j \in \Z_{\geqslant 0}$. This arc (resp. $m$-jet)  factors
through $(Z, 0)$ if for any $F \in I$ we have $
 F(x_1(t), \dots, x_n (t)) =
0$. (resp.  $F(x_1(t), \dots, x_n (t)) = 0 \mod t^{m+1}$).

The arc space $H_Z$ (resp. $m$-jet space $H_{m, Z}$) is the
reduced scheme underlying the affine scheme $\Spec
\mathcal{A}_{Z} $,  where $\mathcal{A}_{Z} = \C [
\underline{a}^{(1)} , \underline{a}^{(2)}, \underline{a}^{(3)},
\dots ] /( \a_F^{(k)} ( \underline{a}^{(1)} , \dots,
\underline{a}^{(k)}))_{k \geqslant 1, F \in I}$ (resp.
 $\Spec \mathcal{A}_{m, Z}$, where
$\mathcal{A}_{m, Z} := \C [ \underline{a}^{(1)} , \dots,
\underline{a}^{(m)}] / ( \a_F^{(k)} ( \underline{a}^{(1)} , \dots,
\underline{a}^{(k)}) )_{ k = 1, \dots, m, F \in I} )$.

We have morphisms of schemes $j_m:H_Z\rightarrow H_{m,Z}$ and
$j_{m}^{m+1} : H_{m+1, Z} \rightarrow H_{m, Z}$ induced by
truncating arcs or jets $\mod t^{m+1}$, for every $m \geqslant 0$.
We have that $H_{Z} = \limproj H_{m, Z}$.

If $h (t) =  \sum_{i \geqslant 0} a_i t^i $ is a formal power
series and $m \geqslant 0$ we set $j_m(h) := h(t) \mod t^{m+1}$.

\section{Some basic definitions on toric geometry} \label{sec-tor}

See \cite{Fu, Ewald, Oda}  for general references on toric
geometry. If $N \cong \Z^{d}$ is a lattice we denote by $N_\R$
(resp. $N_\Q$) the vector space spanned by $N$ over the field $\R$
(resp. over $\Q$). If $\{ u_i \}_{i\in I} \subset N_\R$ we denote by
$ \mathsf{ span}_\Q  \{ u_i \}_{i\in I}$ the linear subspace spanned by the $u_i$ over $\Q$.

In what follows a {\it cone} mean a {\it
rational convex polyhedral cone}: the set of non negative linear
combinations of vectors $a_1, \dots, a_r \in N$. The cone $\t$ is
{\it strictly convex} if it contains no lines,  in that case we
denote by $0$ the $0$-dimensional face of $\t$. The {\em dual
cone}  $\t^\vee$ (resp. {\em orthogonal cone} $\t^\perp$) of $\t$
is the set $ \{ w  \in M_\R\ |\ \langle w, u \rangle \geqslant 0,$
(resp. $ \langle w, u \rangle = 0$) $ \; \forall u \in \t \}$. We
denote by $\stackrel{\circ}{\t}$ or by $\mathsf{ int} (\t)$ the
relative interior of the cone $\t$. A {\em fan} $\Sigma$ is a
family of strictly convex
  cones  in $N_\R$
such that any face of such a cone is in the family and the
intersection of any two of them is a face of each. The relation
$\theta \leqslant \t$  (resp. $\theta < \t$) denotes that $\theta$
is a face of $\t$ (resp. $\theta \ne \t$ is a face of $\t$). The
{\em support} (resp. the $k$-{\em skeleton}) of the fan $\Sigma$
is the set $|\Sigma | := \bigcup_{\t \in \Sigma} \t \subset N_\R$
(resp. $\Sigma^{(k)} = \{ \t \in \Sigma \mid \dim \t = k \}$). We
say that a fan $\Sigma'$ is a {\it subdivision\index{fan
subdivision}} of the fan $\Sigma$ if both fans have the same
support and if every cone of $\Sigma'$ is contained in a cone of
$\Sigma$. If $\Sigma_i$ for $i=1,\ldots,n$ are fans with the same
support their intersection
$\cap_{i=1}^n\Sigma_i:=\{\cap_{i=1}^n\t_i\ |\ \t_i\in\Sigma_i\}$
is also a fan.

Let $\t$ be a strictly convex cone rational for the lattice $N$.
The toric variety $Z_{\t} := \Spec \C [ \t^\vee \cap M] $,
denoted also by $Z_{\t, N}$ or $Z^{\t^\vee \cap M}$, is normal.
The torus $ T_N:= Z^{M}$ is an open dense subset of $Z_\t$, which
acts on $Z_\t$ and the action extends the action of the torus on
itself by multiplication. There is a one to one correspondence
between the faces $\theta $ of $\t$ and the orbits $\orb_\theta$
of the torus action on $Z_\t$,
 which reverses the
inclusions of their closures. The closure of $\orb_\theta$ is the
toric variety $Z^{\s^\vee \cap \theta^\perp \cap M }$  for $\t
\leqslant \s$.   The orbit $\orb_\t$ is reduced to a closed point
called the origin $0$ of the toric variety $Z_\t$. The ring $\C \{
\t^\vee \cap M \} $ of germs of holomorphic functions at $0 \in
Z_\t$ is a subring of the ring $\C [[\t^\vee \cap M]]$ of formal
power series with exponents in $\t^\vee \cap M$.

If  $ \emptyset \ne {\mathcal I} \subset   \t^\vee \cap M $, the
{\em Newton polyhedron} ${\mathcal N} ({\mathcal I})$ of the
monomial ideal defined by ${\mathcal I}$ is the Minkowski sum  of
sets ${\mathcal I} + \t^\vee$. The {\em support function}
$\mbox{\rm ord}_{\mathcal {\mathcal I}}$ of the polyhedron
${\mathcal N} ( {\mathcal I} )$ is defined by $\mbox{\rm
ord}_{\mathcal I} : \t \rightarrow \R$, $ \nu \mapsto \inf_{\omega
\in {\mathcal N} ( {\mathcal I} )} \langle \nu, \omega \rangle$. A
vector $\nu \in \t$ defines the face  $ {\mathcal F}_\nu := \{
\omega \in {\mathcal N} (\mathcal I) \mid \langle \nu, \omega
  \rangle   = \mbox{\rm ord}_{\mathcal I} (\nu ) \}$
of  the polyhedron ${\mathcal N} ({\mathcal I})$.  All faces of
${\mathcal N} ({\mathcal I})$ are of this form, the compact faces
are defined by vectors $\nu  \in \stackrel{\circ}{\t}$. The {\em
dual fan}   $\Sigma ({\mathcal I})$ associated to an integral
polyhedron ${\mathcal N} (\mathcal I)$ is a fan supported on $\t$
which is  formed by the cones $ \s( {\mathcal F} ) := \{ \nu \in
\s \; \mid \langle \nu , \omega \rangle   = \mbox{\rm
ord}_{\mathcal I} ( \nu ),
 \; \forall \omega \in {\mathcal F}\}$,
for ${\mathcal F}$ running through the faces of ${\mathcal N}
({\mathcal I})$. Notice that if $\theta \in  \Sigma ( {\mathcal I}
)$ and if $\nu, \nu' \in \stackrel{\circ}{\theta}$ then we have
that ${\mathcal F}_{\nu} = { \mathcal F}_{\nu'}$ and we denote
this face of ${\mathcal N} (\mathcal I)$ also by  ${ \mathcal
F}_{\theta}$.

The
 affine varieties $Z_\s$ corresponding to cones in a fan $\Sigma$
 glue up to define a {\em toric variety\index{toric variety}} $
 Z_\Sigma$.
The subdivision $\Sigma'$ of  a fan $\Sigma$ defines a {\em toric
 modification\index{modification}} $ \pi_{\Sigma'} : Z_{\Sigma '}
 \rightarrow   Z_\Sigma$.

If $\mathcal{I}$ is a monomial ideal of $Z_\t$ $\Sigma = \Sigma
({\mathcal I}) $,  the toric modification
  $ \p_\Sigma : Z_{\Sigma}   \rightarrow Z_\t$
  is the {\em normalized blowing up} of $Z_\t$ centered at ${\mathcal I}$ (see \cite{LR} for instance).

\section{Quasi-ordinary hypersurface singularities}
\label{secQO}

A germ $(S, 0)$ of complex analytic variety equidimensional of dimension $d$
 is {\em quasi-ordinary} (q.o.) if there exists a
finite projection $\p: (S, 0) \rightarrow (\C^d,0)$ which is a
local isomorphism outside a normal crossing divisor. If $(S, 0)$
is a hypersurface there is an embedding $(S, 0) \subset (\C^{d+1},
0) $, defined by an equation $f= 0$,  where $f \in \C \{ x_1,
\dots, x_d  \} [x_{d+1}]$ is a {\it q.o.~polynomial}: a
Weierstrass polynomial with discriminant $\D_{x_{d+1}} f$ of the
form $\D_{x_{d+1}} f = x^\d \epsilon$ for a unit $\epsilon$ in the
ring $ \C \{ x \}$ of convergent power series in the variables $x=
(x_1, \dots, x_d)$ and $\d \in \Z^d_{\geqslant 0}$.

We suppose that $(S,0)$ is analytically irreducible, that is
  $f \in \C \{ x_1, \dots, x_d  \} [x_{d+1}]$ is irreducible.
The Jung-Abhyankar theorem guarantees that the roots of a q.o.~
polynomial $f$, called {\it q.o.~branches}, are fractional power
series in $\C \{ x^{1/n_0}\}$, for $n_0 =\deg f$ (see
\cite{Abhyankar}).

\begin{Lem} \label{expo}     {\rm (see  \cite{Gau}, Prop. 1.3).}
Let  $f \in \C \{ x_1, \dots, x_n \} [x_{d+1}] $ be an irreducible
q.o.~polynomial. Let $\z$ be a root of $f$ with expansion:
\begin{equation} \label{expan}
\z = \sum \beta_{\l} x^\l.
\end{equation}
There exists $0 \ne \l_1, \dots, \l_g \in \Q^d_{\geqslant 0}$
such that if $M_0 :=\Z^{d} $ and $M_j := M_{j-1} + \Z \l_j$ for
$j=1, \dots, g$, then:
\begin{enumerate}
\item [(i)] $\beta_{\l_i} \ne 0$ and if $\beta_{\l} \ne 0$  then $\l \in M_j$ where
    $j$ is the unique integer such that  $\l_j \leqslant \l$ and
    $\l_{j+1} \nleq \l$ (where $\leqslant$ means coordinate-wise and we convey that $\l_{g +1} = \infty$).
\item [(ii)] For $j=1, \dots, g$, we have   $\l_j \notin M_{j-1}$,   hence the index
$n_j = [M_{j-1} : M_j]$ is $> 1$.
\end{enumerate}
\end{Lem}

\begin{definition}
The exponents $\l_1 , \dots, \l_g $ in Lemma \ref{expo} above
 (resp. the monomials $x^{\l_1}, \dots, x^{\l_g}$) are called {\em characteristic}
of the q.o.~branch $\z$. We denote by  $M$  the lattice $M_g$ and
we call it the lattice associated to the q.o.~branch $\z$. We
denote its dual lattice by $N$. For convenience we denote $\l_0
:=0$.
\end{definition}

Without loss of generality we relabel the variables $x_1, \dots,
x_d$ in such a way that if  $\l_j = ( \l_j^1, \dots, \l_j^d) \in
\Q^d$ for $j=1, \dots, g$, then we have:
\begin{equation} \label{lex}
(\l_1^1, \dots, \l_g^1) \geqslant_{\mbox{\rm lex}} \cdots
\geqslant_{\mbox{\rm lex}} (\l_1^d, \dots, \l_g^d),
\end{equation}
where $\geqslant_{\mbox{\rm lex}}$ is lexicographic order. The
q.o.~branch $\z$
 is normalized if
$\l_1$ is not of the form $(\l_1^1, 0, \dots, 0)$ with $\l_1^1 <
1$. Lipman proved that the germ $(S,0)$ can be parametrized by a
normalized q.o.~branch (see  \cite{Gau}, Appendix).
 We assume from now on that the q.o.~branch $\z$ is normalized.

The semigroup $\Z^d_{\geqslant 0}$ has a minimal set of generators
$e_1,\ldots,e_d$ which is a basis of the lattice $M_0$. The dual
basis of the dual lattice $N_0$ spans a regular cone $\s$ in
$N_{0,\R}$. It follows that $\Z^d_{\geqslant 0} = \s^\vee \cap
M_0$, where $\s^\vee = \R^d_{\geqslant 0}$ is the dual cone of
$\s$. The $\C$-algebra $\C \{  x_1, \dots, x_d \} $ is isomorphic
to $\C \{ \s^\vee \cap M_0 \} $. This isomorphism sends the
monomial $x_1^{\a_1} \cdots x_d^{\a_d}$ in the monomial $X^\a \in
\C \{ \s^\vee \cap M_0 \}$ if $\a = \sum_{i=1}^d \a_i e_i$. The
local algebra ${\mathcal O}_S = \C\{x_1, \dots,
x_d\}[x_{d+1}]/(f)$ of the singularity $(S,0)$ is isomorphic to
$\C\{\s^\vee \cap M_0\}[\z]$. By Lemma \ref{expo} the series  $\z
$ can be viewed as an element $\sum \beta_\l X^\l$ of the algebra
$ \C\{\s^\vee \cap M\}$.

\begin{Lem} \label{normal}
 {\rm (See \cite{GP4}).} The homomorphism ${\mathcal O}_S \rightarrow
\C\{\s^\vee \cap M\}$ is the inclusion of ${\mathcal O}_S$ in
its integral closure in its field of fractions.
\end{Lem}

The previous Lemma shows that the normalization $(\bar{S}, 0)$ of
a q.o.~hypersurface singularity  $(S, 0)$ is the germ of the
toric variety $\bar{S} = Z^{\s^\vee \cap M}$ at the distinguished
point.
 More generally, the
normalization of a q.o.~singularity, non necessarily
hypersurface, is a toric singularity (see \cite{PPP04} and \cite{PP-C}).
If $n: (\bar{S}, 0) \rightarrow (S, 0)$ is the normalization map
the composite
\begin{equation}  \label{composite}
     (\bar{S}, 0) \stackrel{n}{\rightarrow} (S, 0)  \stackrel{\p}{\rightarrow}   (\C^d, 0) ,
\end{equation}
 is a q.o.~projection,
since
it is the toric map defined by the
inclusion of algebras, $\C\{\s^\vee \cap M_0\} \subset \C\{
\s^\vee \cap M \}$,  induced by the finite index lattice extension $M_0 \subset M$ (see \cite{Oda}).

If $\theta \leqslant \s$,  the toric map between the  orbit
closures
 $Z^{\s^\vee \cap M \cap \theta ^\perp} \rightarrow Z^{\s^\vee \cap M_0 \cap \theta ^\perp}$
is the composite of
\begin{equation} \label{tor-sec}
 ( Z^{\s^\vee \cap M \cap \theta ^\perp}, 0)
  \stackrel{n_{\theta }}{\rightarrow} (S_\theta, 0)   \stackrel{\p_{\theta}}{\rightarrow}
 (Z^{\s^\vee \cap M_0 \cap \theta ^\perp}, 0),
\end{equation}
 where $n_{\theta}$ and $\p_{\theta}$ denote   respectively the restrictions of $n$ and $\p$  and
$(S_{\theta}, 0)$ is the coordinate section of $(S, 0)$ given by:
$S_\theta :=S\cap\{x_{j}=0\ |\ \mbox{ for }1\leqslant j\leqslant
d\mbox{ and } e_j\notin \theta^\perp \}$.
It follows that $(S_\theta, 0)$ is q.o.~(see \cite{Lipman2}). The
germ   $(S_\theta, 0)$  is a q.o.~hypersurface of dimension $\dim
\theta^\perp$   parametrized by the series
$\z_\theta := \sum_{\l\in\s^\vee \cap \theta^\perp}\beta_\l X^\l$.

\begin{definition}
For $\theta \leqslant\s$, we call $S_\theta$ the q.o.~coordinate
section associated to $\theta$, we denote by $M(\theta,\z)$ the
lattice associated to the q.o.~branch $\z_\theta$ and by
$N(\theta,\z)$ its dual lattice. \label{notlattQO}
\end{definition}
Notice that the dual cone of $\s^\vee \cap \theta^\perp$ is  equal
to the image $\s/\theta\R$  of the cone $\s$ in the quotient
vector space $N_\R/\theta\R$. We have finite index lattice
extensions $ M_0(\theta):=M_0\cap\theta^\perp \hookrightarrow
M(\theta, \z) \hookrightarrow M(\theta):=M\cap\theta^\perp$.  The
map $n_{\theta}$ in (\ref{tor-sec}) is a ramified covering  with
$[ M\cap\theta^\perp : M(\theta, \z)]$ sheets, which is unramified
over the torus. The index $[ M\cap\theta^\perp :      M(\theta,
\z)]$ is equal to the number of irreducible components of the germ
of $S$ at a generic point of $S_\theta$ (cf.~Proposition 4.5.4 in
\cite{Lipman2}).

The elements of $M$ defined by: $ {\g}_1 =  \l_1$ and $ {\g}_{j+1}
- n_j {\g}_{j} = \l_{j+1} -  \l_{j}$ for $ j= 1, \dots, g-1$, span
the semigroup $\Gamma := \Z^d_{\geqslant 0} +  \g_1 \Z_{\geqslant
0} + \cdots + \g_g \Z_{\geqslant 0} \subset \s^\vee \cap M$. The
semigroup $\Gamma$ defines an analytic invariant of the germ
$(S,0)$ (see \cite{Jussieu, Grenoble, PPP04}).
\begin{definition}      \label{zs}
The {\em monomial variety} associated to $(S, 0)$ is the toric variety $Z(S):= \Spec \C [ \Gamma ]$.
\end{definition}

Following Teissier's approach \cite{Goldin}, \cite{T},  the
singularity $(S,0)$ can be presented, after re-embedding  in a
suitable affine space of larger dimension, as the generic fiber of
a $1$-parametrical deformation with special fiber equal to the
monomial variety $Z(S)$ (see \cite{GP4}). This family is
equisingular in the sense that one toric morphism of the affine
space provides a simultaneous embedded resolution of singularities
of the family (see \cite{Goldin} for the one dimensional case,
\cite{GP4} in the q.o.~case and \cite{T} for related results in a
more general context). The normalizations of $(S,0)$ and of
$(Z(S),0)$ coincide. See \cite{GP4,  PedroNash} for the properties
of the equisingular deformation of $Z(S)$ with generic fiber $(S,
0)$.

\section{Arcs and jets on a quasi-ordinary hypersurface}
\label{ArcsAndJetsQO}

In this Section we study the arcs in the q.o.~hypersurface
$(S,0)$  by using the  toric structure of $(\bar S,0)$, following
the approach of \cite{Ishii-fourier, Ishii-crelle, CoGP}.  We keep
notations of Section \ref{secQO}. Recall that the normalization
$\bar{S}$ of $S$  is equal to the toric variety $Z^{\s^\vee \cap
M}$. The set $H_{\bar{S}}^*$  of arcs of $H_{\bar{S}}$ with
generic point in the torus is \[ H_{\bar{S}}^* =\{ h\in
H_{\bar{S}} \, |\ X^e \circ h\neq 0,\, \forall e \in \s^\vee \cap
M \}.\] Any arc  $\bar{h}  \in H^*_{\bar{S}}$ defines two group
 homomorphisms $\nu_{\bar{h}} : M \rightarrow \Z $ and $\omega_{\bar{h}}: M
\rightarrow \C[[t]]^* $ by \[
 {X}^m \circ {\bar{h}} = t^{\nu_{\bar{h}} (m)}
\omega_{\bar{h}} (m), \mbox{ for } m \in M .\] If $m \in \s^\vee
\cap M$ then we have that $\nu_{\bar{h}} (m)
>0$ hence $\nu_{\bar{h}}$ belongs to $\stackrel{\circ}{\s} \cap N$.
 It follows that  $\omega_{\bar{h}}$
 defines an arc in the torus $T_N$ (see
\cite{CoGP}).
We define similarly the set $H_S^* \subset H_S$:
\begin{definition} (see \cite{Ishii-fourier,Ishii-crelle, CoGP})
\label{estrella} $H_S^*=\{h\in H_S\ |\ x_i\circ h\neq 0,\
i=1,\ldots,d+1\}$.
\end{definition}
The set $  H_S^*$ consists of those arcs $h \in H_S$ such that
there exists $\bar{h} \in H_{\bar{S}}^*$ such that $h=n\circ\bar
h$.  Notice that if $h \in H_S^*$ then the generic point of $h$ is
not contained in the singular locus   $\mbox{{\rm Sing}} \, S$  of
$S$, since $\mbox{{\rm Sing}} \, S \cap \p^{-1} ((\C^*)^d) =
\emptyset$ by definition of $\p$.  By the valuative criterion of
properness applied to the normalization map, there exists a unique
arc $\bar h\in H_{\bar{S}}$ such that $h=n\circ\bar h$. Since
(\ref{composite}) defines a q.o.~projection it follows that
$\bar{h} \in H_{\bar{S}}^*$ and $h \mapsto \bar{h}$ defines a
bijective correspondence between the sets $H_{S}^* $ and
$H_{\bar{S}} ^*$.

\begin{definition}
For any $\nu \in \stackrel{\circ}{\s} \cap N$
we define the sets:
\[
H^*_{\bar{S}, \nu} = \{ \bar{h} \in H^*_{\bar{S}} \mid
\nu_{\bar{h} } = \nu \} \mbox{ and }  H^*_{{S}, \nu} = \{ {h} \in
H^*_{{S}} \mid  \exists \bar{h} \in H^*_{\bar{S}} :\,  n \circ
  \bar{ h}  = h, \, \nu_{\bar{h} } = \nu \}.
\]
\end{definition}

Ishii noticed that the space of arcs in the torus acts on the arc
space of a toric variety (see \cite{Ishii-algebra,Ishii-crelle}).
The set $H_{\bar{S},\nu}^*$ is an orbit of
    the action of the arc space of the torus of $\bar S$
    (see \cite{Ishii-algebra}).
The map $\bar{h} \mapsto n \circ
  \bar{ h}$ defines a bijective correspondence between the sets
$H^* _{\bar{S}, \nu}$ and $H^*_{S, \nu}$. We usually denote the set
$H^* _{S, \nu}$ by $H^* _\nu$.
 The set $H^*_{\bar{S}}$ is invariant for this action of the arc
space of the torus and we obtain the partitions: $H^*_{\bar{S}} =
\bigsqcup_{\nu \in \stackrel{\circ}{\s} \cap N}
H_{\bar{S},\nu}^*$ and
$H_S^*=\bigsqcup_{\nu\in\stackrel\circ\s\cap
    N} \, H_{S,\nu}^*$.

Notice that $H^*_{S}$ can be defined also as the set of arcs $h
\in H_{S}$ such that the arc $\p \circ h \in H_{\A_\C^d}$ has
generic point in the torus, where we consider the affine space
$\A_\C^d$ with the structure of toric variety $\A_\C^d =
Z^{\s^\vee \cap M_0}$. If $h \in H_{S}$ then $\p \circ h$ factors
through a unique minimal orbit closure,  of the form $Z^{ \s^\vee
\cap \theta^\perp \cap M_0 }$,  in such a way that $\p \circ h$ has
generic point in the torus of $Z^{\s^\vee \cap \theta^\perp \cap
M_0 }$. In this case, it follows from the properties of
(\ref{composite}) that
 the  arc $h$ factors
through $(S_\theta, 0)$ and belongs to the set $H_{S_\theta}^*$.
We deduce from this observations the following partition of the
arc space  $H_S=\bigsqcup_{\theta\leqslant\s}H_{S_\theta}^*$.

\section{Statement of the main results}
 \label{mainQO}

We state in this section the main results of the paper. We keep
notations of the previous sections.

\begin{Not} \label{dplusg}
We denote by $e_1, \dots, e_d$ the elements of the canonical basis
of $\Z^d$. We also denote the characteristic exponent $\l_{j}$ of
the q.o.~branch $\z$ by  $ e_{d+j}$,  for $j=1, \dots, g$. We
set $e_0 := \infty$ and $e_{d+g+1}:= \infty$.
\end{Not}

\begin{definition}
We introduce the following subsets of $\s^\vee \cap M$ associated
to the q.o.~branch $\z$ parametrizing the germ $(S,0)$.
\begin{equation}       \label{j-kQO}
\J_k := \{ e_{j_1} + \cdots + e_{j_k}  \  | \ e_{j_1} \y \cdots \y
e_{j_k}  \ne 0, \,   1 \leqslant j_1, \dots, j_{k-1} \leqslant  d,
\mbox{ and } \ 1 \leqslant j_k \leqslant d+g  \}.
\end{equation}
We call the  monomial ideal of $\C \{ \s^\vee \cap M \}$ defined
by (\ref{j-kQO})
 the {\em $k^{th}$-logarithmic
jacobian ideal} of $(S,0)$ relative to the q.o.~projection $\p$,
for $k=1,\dots,d$. We abuse slightly of notation by denoting this
ideal also by $\J_k$. We denote by $\Sigma_k $  (resp. by
$\ord_{\mathcal{J}_k }$) the dual subdivision of $\s$ (resp. the
support function) associated to the Newton polyhedron of the
$k^{th}$-logarithmic jacobian ideal ${\mathcal J}_k $, for $k=1,
\dots, d$.
 \label{idealesJac}
\end{definition}

\begin{remark}
These ideals can be defined in terms of holomorphic differential
forms with logarithmic poles outside the torus of $\bar S$ (see
Section \ref{torQO}). The terminology is inspired by \cite{LR} and
\cite{CoGP}.
\end{remark}

Now we introduce an auxiliary series to study $P_{
{\geom}}^{(S, 0)}(T)$.
 By Greenberg's
Theorem, for any $\theta \leqslant \s$ we have that $j_s
(H_{S_{\theta}})$ is a constructible subset of the $s$-jet space
of $(S,0)$. It follows that
\[ j_s(H_S)\setminus\bigcup_{0\neq\theta\leqslant\s}j_s(H_{S_\theta}^*) =
j_s(H_S)\setminus\bigcup_{0\neq\theta\leqslant\s}j_s(H_{S_\theta})
\] is also a constructible subset. Taking the images of theses
sets in the Grothendieck ring it is natural to consider the
auxiliary Poincar\'e series,
\begin{equation} \label{auxiliarPQO}
P(S):=\displaystyle\sum_{s\geqslant
0}\big[j_s(H_S)\setminus\displaystyle\bigcup_{0
\neq\theta\leqslant\s}
    j_s(H_{S_\theta})\big]T^s \in K_0 ({  \Var}_\C)
    [[T]],
\end{equation}
which  measures the jets of arcs through $(S,0)$ which do not
factor through any proper q.o.~coordinate section $(S_\theta,
0)$, $0 \ne \theta \leqslant \s$. We will show that the
coefficient of $T^s$ in the series $P(S)$ is a finite sum of
classes $[j_s (H^* _\nu)]$. Notice that if  $\t \leqslant \s$ the
series  $P(S_\t)$ is defined similarly by formula
(\ref{auxiliarPQO}).
\begin{Pro} {\rm (cf.~Prop. 4.1 \cite{CoGP})}
 \label{descompPgeomQO}  We have the equality $
P_{
{\geom}}^{(S,0)}(T)=\sum_{\theta\leqslant\s}P(S_\theta)$.
\end{Pro}

\begin{remark}  $\,$ \label{P-codim1QO}
\begin{enumerate}
 \item[(i)]    The  series $P(S_\s)$ takes into account those
jets of arcs in $H_S$ which truncate to $0$. We have that
$P(S_\s)=\sum_{s\geqslant 0}T^s =  (1-T)^{-1}$. \item[(ii)]    If
$\theta\leqslant\s$ is of codimension one then $S_{\theta}$ is a
plane curve. We have that $P(S_\theta)=\frac{\L-1}{1-\L
T}\frac{T^m}{1-T^m}$, where $m$ is the multiplicity of $S_\theta$
(see Proposition 10.2.1  in \cite{DL2}).
\end{enumerate}
\end{remark}

Recall that if $\Sigma$ is a fan then $\Sigma^{(1)}$ denotes the
set of one dimensional cones of $\Sigma$ (see Section
\ref{sec-tor}).
\begin{definition} {\rm (cf.~Notation 4.7 \cite{CoGP})} \label{jotak-bisQO}
 The following maps   are piece-wise linear functions  on the cone $\s$:
\[
\left\{
\begin{array}{lcccccccl}
\f_1    &   :=  &   \mathsf{ ord}_{\mathcal J_1  } &  \mbox{ and }
& \f_k  & :=  &   \mathsf{ ord}_{\mathcal J_k  }  - \mathsf{
ord}_{\mathcal J_{k-1}  } &  \mbox{ for }  k=2,\dots,d,
   \\
 \Psi_1 &:=& 0 & \mbox{ and } &  \Psi_{k} &:= & (k-1) \, \ord_{\J_k}-k \, \ord_{\J _{k-1}} & \mbox{ for }  k=2, \dots,
 d,
\end{array} \right.
\]
We define $\phi_0:=0$ and $\phi_{d+1}:=\infty$ by convenience.
If $\r \subset \s$ is a cone of dimension one, we denote by
$\nu_\r$ the generator of the semigroup $\r \cap N$. We define
the finite set:
\begin{equation} \label{BLQO}
    B(S) : = \left\{ (d,1) \right\}  \cup  \bigcup_{k=1}^d  \left \{ ({\Psi_k}(\nu_\r), { \f_{k}} (\nu_\r))
  \,   \mid \r \in \cup_{i=1}^k\Sigma_i^{(1)},\, \mbox{and } \r \cap
\stackrel{\circ}{\s} \ne \emptyset \mbox{ if }k<d\,
     \right\}.
\end{equation}
This definition applies for the q.o.~sections $S_\theta$, for $0
\leqslant \theta < \s$. For $\theta = \s$ we set $B(
S_\theta) := \{ (0, 1) \}$.
 \end{definition}

\begin{definition} \label{Ak} {\rm (cf.~Def. 5.4 \cite{CoGP})}   For $0 \leqslant k \leqslant d$ we set
\[A_k := \{ (\nu, s) \in (\stackrel{\circ}{\s} \cap N) \times
\Z_{>0} \, \mid \, \f_k (\nu) \leqslant  s < \f_{k+1} (\nu) \}.\]
\end{definition}

\begin{The}\label{keyQO}     {\rm (cf.~Th. 7.1 \cite{CoGP})}
If  $(\nu, s) \in A_k$ then the jet space $j_s( H^*_\nu)$  is a
locally closed subset of $H_s(S)$ isomorphic to $\{ 0 \}$ if $k=0$
or to $
 (\C^*)^{k}
\times \A_\C^{s k - \mathsf{ord}_{\mathcal J_{k}} (\nu)}$ if $1
\leqslant k \leqslant d$.
\end{The}

         Theorem  \ref{keyQO}   is essential to prove the main results of the paper:

\begin{The}  {\rm (cf.~Th. 4.9 \cite{CoGP})}   \label{PLambdaRacQO}
Let $(S,0)$ be an irreducible germ of q.o.~hypersurface. Then
there exists a polynomial $Q_S \in\Z[\L,T]$ determined by the
lattice $M$ and the Newton polyhedra of the logarithmic jacobian
ideals of $(S,0)$ such that
\begin{equation} \label{formaQO}
P(S) =  Q_S \prod_{(a, b) \in B(S)} (1-\L^{a}T^{b})^{-1}.
\end{equation}
\end{The}

Combining Theorem \ref{PLambdaRacQO} with Proposition
\ref{descompPgeomQO} we deduce the following Corollary:

\begin{Cor}  {\rm (cf.~Cor. 4.10 \cite{CoGP})}  \label{P-geomQO}
We have that
\[
P_{\geom}^{(S, 0)}(T) = \sum_{\theta \leqslant \s}
Q_{S_\theta} \prod_{ (a, b) \in B( S_\theta) }
(1-\L^{a}T^{b})^{-1}.
\]
The series $P_{\geom}^{(S, 0)}(T)$ depends only on the logarithmic jacobian ideals and lattices $M
(\theta, \z)$ associated to the q.o.~sections
$S_\theta$ for $0 \leqslant \theta \leqslant \s$ (with respect to
the q.o.~projection $\p$).
\end{Cor}

As an application we  give a formula for the   {\em motivic
volume} of the arc space of a q.o.~hypersurface. We refer to
\cite{DL1, DL-M, Loo} for the  definition
of {\em measurable sets} and properties of the {\em motivic
volume}.

Let $\t$ be a strictly convex cone rational for the lattice $N$.
The {\em generating series}  $F_{\stackrel\circ\t\cap
N}(x):=\sum_{\nu\in\stackrel\circ\t\cap N}x^\nu$ has  a rational
form \[  F_{\stackrel\circ\t\cap N}(x)=R_{\t \cap N}
\prod_{\r\leqslant\t,\ \dim\r=1}(1-x^{\nu_\r})^{-1} \mbox{ with }
R_{\t \cap N} \in \Z[ \t \cap N] \]  (see \cite{CoGP} for
instance).
 If $\t \in \Sigma_d$  we denote by $\eta_\t :\Z[\t\cap N]\rightarrow\Z[\L^{\pm 1}]$ the toric map
given by $\eta_\t (x^\nu)=\L^{-\ord_{\J_d}(\nu)}$.

\begin{Cor}   {\rm (cf.~Prop. 10.1 \cite{CoGP})}
The motivic volume of the arc space $H_S$ of the q.o.~
hypersurface $(S, 0)$  is equal to
\[\mu(H_S)=(\L-1)^d  \sum_{{\t\in\Sigma_d}}^{\stackrel\circ\t\cap\stackrel\circ\s\neq\emptyset}
\eta_\t (R_{\stackrel\circ\t\cap N})  \prod_{\r\leqslant\t,\
\dim\r=1}(1-\L^{-\ord_{\J_d}(\nu_\r)})^{-1}       .\]
\label{vol-mot2QO}
\end{Cor}

\begin{remark} In Section \ref{QOexample} we give an example of q.o.
surface singularity such that all the  candidate poles in
Corollary \ref{P-geomQO} are actual poles.
\end{remark}

We give a geometrical interpretation of the set of candidate poles
of the series $P_{\geom}^{(S,0)} (T)$.

\begin{definition}
 For $1 \leqslant k \leqslant d$ we denote by $\p_{k}$ the composite of the
 normalization map $\bar S=Z^{\s^\vee\cap M} \rightarrow S$ with the toric
 modification of $Z_k \rightarrow Z^{\s^\vee\cap M}$ defined by the subdivision $\cap_{i=1}^k
 \Sigma_i$ of $\s$.
 \end{definition}

 The modification $\p_k$ is the minimal modification with normal
 source, which factors through the normalization of $S$
 and the normalized blowing up of $\bar{S}$  with center $\J_i$, for $i=1, \dots, k$.
 The rays $\r$ in the fan $\cap_{i=1}^k \Sigma_i$ correspond bijectively to
orbit closures of $Z_k$ which are of codimension one. If $\nu_\r$
is the generator of the semigroup $\r \cap N$ we denote by
$E_{\nu_\r}$ the irreducible component corresponding to $\r$. We
denote by $\mbox{val}_{{\nu_\r}}$ the divisorial valuation of the
field of fractions of $\C \{ \s^\vee\cap M \} $, which is
associated to the divisor $E_{\nu_\r}$. If $m \in M$ then we have
that
\begin{equation} \label{val1}
\mbox{val}_{\nu_\r} (X^m) = \langle \nu_\r , m \rangle.
\end{equation}
We have that $\stackrel{\circ}{\r} \subset \stackrel{\circ}{\s}$
if and only if $E_{\nu_\r}$ is a codimension one irreducible
component of the exceptional fiber of $\p_k ^{-1} (0) $. If $ 1
\leqslant i \leqslant k \leqslant d$ the pull-back $\p_k^* (\J_i)$
of $\J_i$ by $\p_k$ is a locally principal monomial
ideal sheaf on the toric variety $Z_k$ and by (\ref{val1}) we have that
\[
 \mbox{val}_{\nu_\r} (\p_k^* (\J_i)) = \ord_{\J_i} (\nu_\r).
\]
\begin{Pro} \label{val3}
For $1 \leqslant k \leqslant d$ we have that \[ \mathcal{L}_k:=
(\p_k^* (\J_k))^{k-1} / (\p_k^* (\J_{k-1}))^{k} \mbox{ and }
\mathcal{Q}_k := \p_k^* (\J_k)  / \p_k^* (\J_{k-1})\] define
 locally principal monomial ideal sheaves on $Z_k$ such that
\[
\begin{array}{lcl}
B(S)  & = &  \{ (d,1) \} \cup \bigcup_{k=1}^{d-1} \{
  ( \mbox{val}_{\nu_\r} ( \mathcal{L}_k ) , \mbox{val}_{\nu_\r}
  (\mathcal{Q}_k) ) \mid E_{\nu_\r} \subset \p_k ^{-1} (0) \}
  \\
 &  &  \cup   \{ ( \mbox{val}_{\nu_\r} ( \mathcal{L}_d ) , \mbox{val}_{\nu_\r}
  (\mathcal{Q}_d) ) \mid  \r \in \cap_{i=1}^d
  \Sigma_i, \dim \r = 1 \}.
\end{array}
\]
\end{Pro}

\section{Combinatorial convexity properties of the logarithmic jacobian ideals}
 \label{conv-newtonQO}

In this section we give a series of results on the properties of
the support functions of the logarithmic jacobian ideals
associated to a q.o.~hypersurface germ $(S,0)$, which are used in Sections \ref{tor-jetQO} and \ref{GenptOrbitsQO}.

\begin{Not}
If $\nu\in\s\cap N$ we denote by $\leqslant_\nu$ the partial order
on $M$ defined by $\l_{\leqslant_\nu}\l'$ if
$\langle\nu,\l\rangle\leqslant\langle\nu,\l'\rangle$.
\label{notationorder}
\end{Not}

\begin{remark}
The logarithmic jacobian ideals $\bar{\J}_1, \dots, \bar{\J}_d$ of
the normalization $(\bar{S},0)$, which are studied in \cite{CoGP},
are different than those of $(S,0)$ in general. Recall that we
have $\bar{S} = Z^{\s^\vee \cap M}$.  If $v_1, \dots, v_n$ are the
minimal sequence of generators of the semigroup $\s^\vee \cap M$
then  we have that $\bar{\J}_k = (X^{v_{j_1} + \cdots + v_{j_k} }
)_{ v_{j_1} \y \cdots \y v_{j_k} \ne 0}$. The combinatorial
convexity properties of the support functions, $ \ord_{\bar{\J}_k}
$ for $k =1, \dots,d$,   are simpler in the toric case. Given $\nu
\in \stackrel{\circ}{\s} \cap N$, up to relabeling, we can assume
for simplicity that $v_1 \leqslant_{\nu} \cdots \leqslant_{\nu}
v_n$. We define inductively $i_1 :=1$ and $i_k := \min \{ i \mid
v_{i_1} \y \cdots \y   v_{i_{k-1}} \y v_i \ne 0 \}$ for $k =2,
\dots, d$. Then we have that $ \ord_{\bar{\J}_k}(\nu) :=
\sum_{l=1}^k \langle \nu, v_{i_l} \rangle$ and $
\ord_{\bar{\J}_{k+1}}(\nu) - \ord_{\bar{\J}_k}(\nu) = \langle\nu,
v_{i_l} \rangle$.
\end{remark}

 The following example shows that the
q.o.~case is not as simple as the toric case.
\begin{Exam} \label{ExamAlg}
Consider a q.o.~branch with characteristic
exponents $\l_1=(1/2,1/2,0)$ and $\l_2=(1/2,1/2,1/4)$.
The points $\nu_1=(4,2,8)$ and $\nu_2=(4,2,4)$
belong to the lattice $N$. It is easy to check that
$\ord_{\J_2}(\nu_1)=\langle\nu_1,e_2+\l_1\rangle$,
$\ord_{\J_2}(\nu_2)=\langle\nu_2,e_2+\l_1\rangle$,
$\ord_{\J_3}(\nu_1)=\langle\nu_1,e_1+e_2+\l_2\rangle$ and
$\ord_{\J_3}(\nu_2)=\langle\nu_2,e_2+\l_1+e_3\rangle$. Then we
get $ (\ord_{\J_3} - \ord_{\J_2})
(\nu_1)=\langle\nu_1,\l_2-\l_1+e_3\rangle$  while $ (\ord_{\J_3} -
\ord_{\J_2})(\nu_2)=\langle\nu_2,e_3\rangle$.
\end{Exam}

\begin{Not}
We fix a partial order on the set $\{e_1,\ldots,e_{d+g}\}$ by
\begin{equation}
e_{i_1}\leqslant_\nu\ldots\leqslant_\nu e_{i_{d+g}}
\label{ordenQO}
\end{equation}
in such a way that if
$\langle\nu,e_{i_j}\rangle=\langle\nu,e_{i_k}\rangle$ for
$1\leqslant i_j\leqslant d$ and $d+1\leqslant i_k\leqslant d+g$,
then $j<k$.
\end{Not}
By Lemma \ref{expo}  we have the inequalities
\begin{equation} \label{ch-order2}
\langle\nu,\l_1\rangle<\langle\nu,\l_2\rangle<\cdots<\langle\nu,\l_g\rangle
\end{equation}

In the following Proposition we use the convention
$\langle\nu,\infty\rangle=\infty>r$, for $r\in\R$ (see Notations
\ref{dplusg}).
\begin{Pro}     {\rm (cf.~Prop. 5.1 \cite{CoGP})}
Let $\nu\in\stackrel\circ\s\cap N$. With respect to the order
(\ref{ordenQO}), the vector $\nu$ defines a sequence $1\leqslant
j_1^{(k)}<\cdots<j_k^{(k)}\leqslant d+g$ such that
$w_k:=e_{j_1^{(k)}}+\cdots+e_{j_k^{(k)}}\in\J_k$ and
$\ord_{\J_k}(\nu)=\langle\nu,w_k\rangle$.  We set
$w_1=e_{j_1^{(1)}}:=e_{i_1}$. Suppose that
$w_k=e_{j_1^{(k)}}+\cdots+e_{j_r^{(k)}}$ is already defined. Set:
\begin{enumerate}
\item[(i)] $\ell_k(\nu) := \mathsf{ span}_\Q \{
    e_{j_{1}^{(k)}}, \dots , e_{j_{k}^{(k)}} \}.$

\item[(ii)] $n(k):=\left\{\begin{array}{lll}
                     0 & \mbox{ if } & \{e_{j_1^{(k)}},\ldots,e_{j_k^{(k)}}\} \setminus \{e_1,\ldots,e_d\}=\emptyset,\\
                     n & \mbox{ if } & \{e_{j_1^{(k)}},\ldots,e_{j_k^{(k)}}\}\setminus\{e_1,\ldots,e_d\}= \{ \l_n \}.\\
                     \end{array} \right.$

\item[(iii)] $t(k):=\min \{1\leqslant j \leqslant g+ 1\ |\
    \l_j \notin\ell_k(\nu)\} $.

\item[(iv)] $m(k):=\left\{\begin{array}{lll}
                    0 & \mbox{ if } &     (\{e_1,\ldots,e_d\}\cap\ell_k(\nu))
                    \setminus\{e_{j_1^{(k)}},\ldots,e_{j_k^{(k)}}\}= \emptyset  \\
                    m & \mbox{ if } &  (\{e_1,\ldots,e_d\}\cap\ell_k(\nu))
                    \setminus\{e_{j_1^{(k)}},\ldots,e_{j_k^{(k)}}\}=\{e_m\}.\\
                    \end{array}\right.$

\item[(v)] $i(k):=\min\{1\leqslant i\leqslant d+g\ |\
    w_k+e_i\in\J_{k+1}\}$, for $k=1,\ldots,d-1$.
\end{enumerate}

Set $a_{k+1}:=w_k+e_{i(k)}$ and
$b_{k+1}:=w_k-\l_{n(k)}+\l_{t(k)}+e_{m(k)}$. Then we have that:
\begin{equation}
w_{k+1}:=\left\{\begin{array}{ll}
                  a_{k+1} & \mbox{ if }\langle\nu,a_{k+1}\rangle\leqslant\langle\nu,b_{k+1}\rangle,\\
                  b_{k+1} & \mbox{ otherwise.}\\
                  \end{array}\right.
\label{rithm}
\end{equation}
\label{PruebaAlgQO}
\end{Pro}
\begin{proof}We prove it by induction on $k$. For $k=1$ the assertion holds, since $\langle\nu,e_{i_1}\rangle={
\min}\{\langle\nu,e_i\rangle\ |\ 1\leqslant i\leqslant
d+1\}$. We suppose that the statement is true for
$k$ and we prove it for $k+1$. By induction hypothesis we have
$w_k=e_{j_1^{(k)}}+\cdots+e_{j_k^{(k)}}\in\J_k$ and $\mathsf{
ord}_{\J_k}(\nu)=\langle\nu,w_k\rangle$. If $n(k)=0$ the argument
coincides with the proof of Proposition 5.1 in \cite{CoGP}. We
assume that $n(k)\neq 0$ and that there exists
$w_{k+1}'\in\J_{k+1}$ different from $w_{k+1}$ and such that
\begin{equation}
\langle\nu,w_{k+1}'\rangle<\langle\nu,w_{k+1}\rangle.
\label{HipAlgQO}
\end{equation}
 \begin{Assertion}
       If the vector $e_j$ with $1\leqslant j\leqslant d$
and $j\neq m(k)$ appears in the expansion of $w_{k+1}'$ as sum of
$k+1$ linearly independent elements of $\{e_1,\ldots,e_{d+g}\}$,
then $e_j$ appears in the expansion of $w_{k+1}$.
 \end{Assertion}
{\em Proof of the assertion}. If $e_j$ does not appear in the
expansion of $w_{k+1}$, then $w_{k+1}'-e_j$ belongs to $\J_k$ and
we deduce
$\langle\nu,w_k\rangle\leqslant\langle\nu,w_{k+1}'-e_j\rangle<\langle\nu,w_{k+1}-e_j\rangle$,
where the first inequality follows by the induction hypothesis and
the second by (\ref{HipAlgQO}). Then we get
$\langle\nu,w_k+e_j\rangle<\langle\nu,w_{k+1}\rangle$ and since
$j\neq m(k)$ the vector $w_k+e_j$ belongs to $\J_{k+1}$, but this
is in contradiction with the choice of $w_{k+1}$ in the algorithm,
hence the assertion holds.      \end{proof}

Now we distinguish various cases:
\begin{enumerate}
 \item[(i)] If $m(k)=0$, then we obtain $w_{k+1}=w_k+e_{i(k)}$ by definition.
 By the Assertion there is  an integer $1\leqslant r\leqslant g$ such that $w_{k+1}'-w_{k+1}=\l_r-\l_{n(k)}$.
 By (\ref{ch-order2}) and (\ref{HipAlgQO}) we deduce that $r<n(k)$, but
 then $w_{k+1}'-e_{i(k)}=w_k-\l_{n(k)}+\l_r\in\J_k$ and $\langle\nu,w_{k+1}'-e_{i(k)}\rangle<\langle\nu,w_k\rangle$,
 This is a contradiction with the induction hypothesis.

 \item[(ii)]  If $m(k) \ne 0$ and $w_{k+1}=w_k+e_{i(k)}$ and if $e_{m(k)}$ does not appear in
        the expansion of $w_{k+1}'$ we apply the
        argument of case (i) to get a contradiction.

    \item[(iii)] If $m(k) \ne 0$, $w_{k+1}=w_k+e_{i(k)}$ and  if $e_{m(k)}$ appears in the
        expansion of $w_{k+1}'$ then by the Assertion
        we have that $w_{k+1}=w+\l_{n(k)}+e_{i(k)}$ and
        $w_{k+1}'= w+\l_r+e_{m(k)}$, where
        $w=w_k-\l_{n(k)}\in\J_{k-1}$. By definition
        $e_{m(k)} \in \ell_k(\nu)$, hence $\l_r$ does not belong to $\ell_k(\nu)$ since
$w_{k+1}'\in\J_{k+1}$. By definition of $t(k)$ we deduce that
$r\geqslant t(k)>n(k)$. Then it follows that
$\langle\nu,w_{k+1}'\rangle\geqslant\langle\nu,w_k-\l_{n(k)}+\l_{t(k)}+e_{m(k)}\rangle\geqslant\langle\nu,w_{k+1}\rangle$,
which contradicts (\ref{HipAlgQO}).

\item[(iv)]  If $m(k) \ne 0$ and
$w_{k+1}=w_k-\l_{n(k)}+\l_{t(k)}+e_{m(k)}$
    the assertion implies that
    $e_{m(k)}$ appears in the expansion of $w_{k+1}' \in\J_{k+1}$. We deduce that
     $w_{k+1}=w+\l_{t(k)}+e_{m(k)}$ and
    $w_{k+1}'=w+\l_r+e_{m(k)}$. Formula
    (\ref{HipAlgQO}) implies  that $r<t(k)$ which is a
    contradiction with the definition of $t(k)$.\hfill
    $\Box$
\end{enumerate}
The two choices appearing in (\ref{rithm})  occur (see Example
\ref{ExamAlg}).

\begin{remark}
  \label{op2}
Notice that  $m(k) > 0$ implies  $n(k)>0$  and $\langle \nu,
\l_{n(k)} \rangle < \langle \nu, e_{m(k)} \rangle$ (otherwise
$w_k' = w_k - \l_{n(k)} + e_{m(k)} \in \J_k$ would verify that
$\langle \nu, w_k' \rangle  <  \ord_{\J_k} (\nu)$).
\end{remark}

\begin{Lem} \label{c1}  {\rm (cf.~Lemma 5.3 \cite{CoGP})}
For all $\nu\in\stackrel\circ\s\cap N$ we have that \[ \f_1 (\nu)
\leqslant \f_2 (\nu) \leqslant \cdots \leqslant \f_d(\nu).\]
\end{Lem}

\begin{proof}The assertion is equivalent to the inequality
$\langle\nu,w_k-w_{k-1}\rangle\leqslant\langle\nu,w_{k+1}-w_k\rangle$
for $1\leqslant k<d$, where $w_{k-1},w_k$ and $w_{k+1}$ are
defined by the algorithm and $w_0:=0$. We distinguish the cases:

\begin{enumerate}
\item[(i)] If $\phi_k(\nu)=\langle\nu,e_{i(k-1)}\rangle$ and
$\phi_{k+1}(\nu)=\langle\nu,e_{i(k)}\rangle$
 the result follows by definition of $e_{i(k-1)}$ in Proposition \ref{PruebaAlgQO}.

\item[(ii)] If  $\phi_k(\nu)=\langle\nu,e_{i(k-1)}\rangle$ and
$\phi_{k+1}(\nu)=\langle\nu,\l_{t(k)}-\l_{n(k)}+e_{m(k)}\rangle$,
 we get $w_k= w_{k-1}+e_{i(k-1)}$ and $w_{k+1} =w_k-\l_{n(k)}+\l_{t(k)}+e_{m(k)}$.
 We distinguish two subcases below:
 \begin{enumerate}
\item [(ii.1)] If $e_{i(k-1)}=\l_{n(k)}$ we get
    $\langle\nu,e_{m(k)}\rangle>\langle\nu,\l_{n(k)}\rangle$  by the definition of
    the algorithm
    hence we deduce
    $\phi_k(\nu)=\langle\nu,\l_{n(k)}\rangle<\langle\nu,\l_{t(k)}-\l_{n(k)}+e_{m(k)}\rangle=\phi_{k+1}(\nu)$ since
    $t(k)>n(k)$.
\item [(ii.2)] If $e_{i(k-1)}\neq\l_{n(k)}$ we obtain
    $1\leqslant i(k-1)\leqslant d$ and
$\phi_k(\nu)=\langle\nu,e_{i(k-1)}\rangle\leqslant
\langle\nu,e_{m(k)}\rangle<\langle\nu,\l_{t(k)}-\l_{n(k)}+e_{m(k)}\rangle=\phi_{k+1}(\nu)$.
\end{enumerate}

\item[(iii)] If
$\phi_k(\nu)=\langle\nu,\l_{t(k-1)}-\l_{n(k-1)}+e_{m(k-1)}\rangle$
and $\phi_{k+1}(\nu)=\langle\nu,e_{i(k)}\rangle$ then  we get
$w_k=w_{k-1}-\l_{n(k-1)}+\l_{t(k-1)}+e_{m(k-1)}$ and $w_{k+1}
=w_k+e_{i(k)}$. Since the vector $e_{i(k)}$ is not a
characteristic exponent we deduce that $w_{k-1}+e_{i(k)}\in\J_k$.
Then we get the inequalities
$\langle\nu,w_k\rangle\leqslant\langle\nu,w_{k-1}+e_{i(k)}\rangle$,
and it follows that:
$\phi_k(\nu)=\langle\nu,\l_{t(k-1)}-\l_{n(k-1)}+e_{m(k-1)}\rangle\leqslant\langle\nu,e_{i(k)}\rangle=\phi_{k+1}(\nu)$.

 \item[(iv)] If $\phi_k(\nu)=\langle\nu,\l_{t(k-1)}-\l_{n(k-1)}+e_{m(k-1)}\rangle$ and
 $\phi_{k+1}(\nu)=\langle\nu,\l_{t(k)}-\l_{n(k)}+e_{m(k)}\rangle$ we obtain
 $w_k=w_{k-1}-\l_{n(k-1)}+\l_{t(k-1)}+e_{m(k-1)}$ and $w_{k+1}=w_k-\l_{n(k)}+\l_{t(k)}+e_{m(k)}$.
 Since the vector $w_{k-1}+e_{m(k)}$ belongs to $\J_k$ we deduce
  $\langle\nu,w_k\rangle<\langle\nu,w_{k-1}+e_{m(k)}\rangle$. We get
   $\phi_k(\nu)=\langle\nu,\l_{t(k-1)}-\l_{n(k-1)}+e_{m(k-1)}\rangle<
   \langle\nu,e_{m(k)}\rangle<\langle\nu,\l_{t(k)}-\l_{n(k)}+e_{m(k)}\rangle=\phi_{k+1}(\nu)$.
\end{enumerate}
\end{proof}

\begin{definition}     \label{ele}
If $(\nu, s) \in  \stackrel{\circ}{\s} \times \Z_{\geqslant 0}$
there is a unique integer $0 \leqslant k \leqslant d$ such that
$(\nu, s) \in A_k$. We set
\[
 \ell_{\nu}^s := \mathsf{span}_\Q \{ e_j \mid 1 \leqslant j < d + t(k), \, \langle \nu, e_j \rangle \leqslant s \}.
\]
   (cf.~Definition \ref{Ak},  Proposition \ref{PruebaAlgQO} and Lemma \ref{c1}).
\end{definition}

\begin{Lem} \label{crucial}            {\rm (cf.~ Lemma 5.7 \cite{CoGP})}
If $(\nu, s) \in A_k$  and $w_k'  = e_{r_1}
+ \cdots + e_{r_k} \in {\mathcal
  J}_{k}$ is a vector such that
 $ \mathsf{ ord}_{{\mathcal J}_{k}} (\nu)  = \langle \nu, w_k'
\rangle $,  then we have
\[  \ell_\nu^s =  \ell_k(\nu) =  \mathsf{
span}_\Q \{ e_{r_1}, \dots , e_{r_k} \}.\]
 In addition, if $1
\leqslant j < n(k)$ and
 $\l_j = \sum_{i=1}^d \l_{j, i } e_i $ then $\l_{j,i} \ne 0$
 implies that $\langle \nu, e_i \rangle \leqslant \langle \nu, \l_j
 \rangle \leqslant s$.
\end{Lem}
\begin{proof}Let $w_k =\sum_{r=1}^k e_{j_r^{(k)}}$ be the vector defined by Proposition
\ref{PruebaAlgQO}. We denote by $\ell = \ell_k (\nu) $ (resp. by
$\ell'$) the linear subspace of $M_\Q$ spanned by the vectors in
the expansion of $w_k$ (resp. $w_k'$).
We prove first that $\ell =  \ell'$. If $\ell'$ and $\ell$ are
distinct we verify that $\phi_k(\nu)=\phi_{k+1}(\nu)$. Suppose that
there exists a vector $e_{j_0}$ appearing in the expansion of
$w_k'$ and such that $e_{j_0}\notin\ell$. We distinguish two
cases:

\begin{enumerate}
\item[(i)] If $1\leqslant j_0\leqslant d$ then the vector
    $w_{k+1}':=w_k+e_{j_0}$ belongs to $\J_{k+1}$.
    We get $\mathsf{
    ord}_{\J_{k+1}}(\nu)\leqslant\langle\nu,w_{k+1}'\rangle=\mathsf{
    ord}_{\J_k}(\nu)+\langle\nu,e_{j_0}\rangle$ hence
     $\phi_{k+1}(\nu)\leqslant\langle\nu,e_{j_0}\rangle$.
    If $\langle\nu,e_{j_0}\rangle>\phi_k(\nu)$ then the
vector $w_{k-1}':=w_k'-e_{j_0}$ belongs to $\J_{k-1}$ and we find
the contradiction $\langle\nu,w_{k-1}'\rangle=\mathsf{
ord}_{\J_k}(\nu)-\langle\nu,e_{j_0}\rangle<\mathsf{
ord}_{\J_k}(\nu)-\phi_k(\nu)=\mathsf{ ord}_{\J_{k-1}}(\nu)$. Hence
we obtain  $\phi_k (\nu) \geqslant\langle\nu,e_{j_0}\rangle$ thus
$\phi_k(\nu)=\phi_{k+1}(\nu)$ holds by Lemma \ref{c1}.

\item [(ii)] If for any $1\leqslant j\leqslant d$, with $e_j$
appearing in the expansion of
    $w_k'$ then
    $e_j \in\ell$, then we get  $w_k =  w+e_r$
    and $w_k'=  w+\l_{n'}$, for $1\leqslant r\leqslant d+g$ and
    $1\leqslant n'\leqslant g$. Then
    $\langle\nu,\l_{n'}\rangle=\langle\nu,e_r\rangle$ and
    by (\ref{ch-order2}) the vector $e_r$ can not
    be a characteristic exponent. Thus $n(k)=0$ and
    we get $w_{k+1}=w_k+\l_{n'}$.
We obtain
    $\phi_{k+1}(\nu)=\langle\nu,\l_{n'}\rangle=\langle\nu,e_j\rangle\leqslant\phi_k(\nu)$
     thus
$\phi_k(\nu)=\phi_{k+1}(\nu)$ holds by Lemma \ref{c1}.
\end{enumerate}

      We obtain that $\ell  \subset \ell_\nu^s$ by checking that
 $\langle\nu,e_{j_r^{(k)}}\rangle\leqslant
s$ for $r=1,\ldots,k$.
    If $k=1$ the inequality is trivial. If $k>1$
and $e_{j_r^{(k)}}$ is a vector appearing in the expansion of
$w_k$ then there are two possibilities.
\begin{enumerate}
\item [(i)]  If there is $ 1 \leqslant l \leqslant k-1 $
    such that $w_{l+1}=w_l+e_{j_r^{(k)}}$ then we get
    $\phi_{l+1}(\nu)=\langle\nu,e_{j_r^{(k)}}\rangle$.
    By Lemma \ref{c1}
    we deduce  $\langle\nu,e_{j_r^{(k)}}\rangle\leqslant s$.

\item [(ii)] If there is $ 1 \leqslant l \leqslant k-1 $
    such that $w_{l+1}=w_l-\l_{n(l)}+\l_{t(l)}+e_{m(l)}$,
    with either $\l_{t(l)}=e_{j_r^{(k)}}$ or
    $e_{m(l)}=e_{j_r^{(k)}}$ then  in both cases we get
    $\phi_{l+1}(\nu)\geqslant\langle\nu,e_{j_r^{(k)}}\rangle$,
    since
    $\langle\nu,\l_{n(l)}\rangle<\langle\nu,e_{m(l)}\rangle$
    and $n(l)<t(l)$.
\end{enumerate}

We prove that $\ell_{\nu}^s \subset \ell$.   If $\langle \nu, e_j
\rangle \leqslant s$ for $1 \leqslant j \leqslant d$ and if $e_j
\notin \ell$ then $\widetilde w_{k+1}:=w_k+e$ belongs to
          $\J_{k+1}$; we deduce that
          $\phi_{k+1}(\nu)\leqslant\langle\nu,\widetilde
          w_{k+1}-w_k\rangle\leqslant s$ contradicting the hypothesis.
We have also shown that $\l_{n(k)} \in \ell$ hence if $n(k) < j <
t(k)$ then $\l_j \in \ell$ by definition of $t(k)$. If $1
\leqslant j < n(k)$ then $\langle \nu, \l_j \rangle    < \langle
\nu, \l_{n(k)} \rangle \leqslant s$ by (\ref{ch-order2})   and it
is easy to see from the algorithm in Proposition \ref{PruebaAlgQO}
that $\l_j$ belongs to $\mathsf{span}_\Q \{ e_i \mid 1 \leqslant i
\leqslant d, \langle \nu , e_i \rangle  \leqslant s \} \subset
\ell$. It follows that $\ell_{\nu}^s \subset \ell$.

For the last assertion notice that if $\l_{j, i} \ne 0$ and if
$\langle \nu, e_i \rangle > \langle \nu, \l_j \rangle$ then we get
a contradiction since $w_k': = \l_j + \sum_{r=1, \dots, d, r \ne
i, m(k)}^{\langle \nu, e_r \rangle \leqslant s} e_r  $ belongs to
$\J_k$ and verifies that $\langle \nu, w_k' \rangle < \langle \nu,
w_k \rangle$.
 \hfill $\    {\Box}$

\begin{Lem} \label{lema-obs}
If $(\nu, s) \in A_k$ we have the following inequalities
\begin{enumerate}
 \item[(i)]      $
s < \langle \nu, \l_{t(k)}  \rangle $   if $n(k) =0$ or if $n(k),
m(k) \ne 0 $ and $\langle \nu, e_{m(k)} \rangle \leqslant s$.

\item[(ii)]
$s < \langle \nu, e_{m(k) } + \l_{t(k)} - \l_{n(k)}  \rangle$ if $n(k) \ne 0$.
\end{enumerate}
\end{Lem}
{\em Proof.} If  $n(k) =0$ then the vector $\l_{t(k)} +
e_{j_1^{(k)}} + \cdots +  e_{j_1^{(k)}} $ belongs to $\J_{k+1}$
hence $\mathsf{ ord}_{{\mathcal J}_{k+1}} (\nu)  \leqslant \langle
\nu, \l_{t(k)} \rangle +
    \mathsf{ ord}_{{\mathcal J}_{k}} (\nu)$.
We deduce from this that  $\f_{k+1} (\nu) \leqslant \langle \nu,
\l_{t(k)} \rangle $. If $n(k) , m(k) \ne 0$ then we get
$\phi_{k+1}(\nu)\leqslant\langle \nu, e_{m(k) } + \l_{t(k)} -
\l_{n(k)} \rangle$ by proof of Lemma \ref{c1}. This implies that
$s < \langle \nu, \l_{t(k)} \rangle$ since $\langle \nu,  e_{m(k)}
-\l_{n(k)}\rangle >0 $ by Remark \ref{op2}.
\end{proof}

\begin{remark}
If $(\nu, s) \in A_k$,  $n(k), m(k) \ne 0$  it may happen that
$\langle \nu, \l_{t(k)} \rangle  \leqslant s$. For instance,
consider a q.o.~branch with characteristic exponents $ \l_1 :=
(\frac{1}{2}, \frac{1}{8},0)$ and $\l_{2} := ( \frac{1}{2},
\frac{1}{8},\frac{1}{18})$. If $\nu = (12, 16, 18) \in N$ then we
get  $\f_1 (\nu) = 8$, $\f_2 (\nu) = 12$ and $\f_3(\nu)=17$. If $12 \leqslant s <
17$ then we obtain  $(\nu,s) \in A_2$, $n(2) =1$, $m(2) = 2$,
$t(2) =2$ and $\langle \nu, \l_2 \rangle = 9 \leqslant s$.
\end{remark}

\begin{definition}
\
\begin{enumerate}

\item [(i)] If $\eta  =\eta_{1}e_1+\cdots+\eta_{d}e_d \in M$ we
denote by
    $\ell(\eta)=\mbox{span}_\Q\{e_i\ |\ \eta_{i}\neq 0\}$
the smallest coordinate subspace containing $\eta$.

\item[(ii)] With respect to the fixed vector
$\nu\in\stackrel\circ\s\cap
    N$, we set
$
    e_{q_{\eta}}=\max_{\leqslant_\nu}\{\{e_1,\ldots,e_d\}\cap\ell(\eta)\}$.
\end{enumerate}
 \label{deflandaq}
\end{definition}

\begin{remark} \label{nk}
It is easy to see from the algorithm in Proposition
\ref{PruebaAlgQO} that if  $1\leqslant j<n(k)$ then
$\ell(\l_j)\subset\ell_\nu^s$ and if $e_{m(k)}\neq \infty$ then
$q_{\l_{n(k)}}=m(k)$.
\end{remark}

\begin{definition}
Let $(\nu,s)\in A_k$ for some $k, s >0$. With the notations of
Proposition \ref{PruebaAlgQO} we define the integer $p(k)$ by
\[ p(k):=\displaystyle\left\{\begin{array}{ll}
\max \{ 0\leqslant j\leqslant  g \mid
\langle\nu,\l_j\rangle\leqslant s \} & \mbox{ if }n(k)=0,
\\
\max ( \{ n(k) \} \cup  \{ n(k) <  j<t(k) \mid \langle\nu,
e_{m(k)} -\l_{n(k)}+ \l_j \rangle\leqslant s \} ) & \mbox{ if }
n(k)>0 .
\end{array}\right.
\]

\label{defpk}
\end{definition}

\begin{remark}
Notice that $\langle\nu,\l_{p(k)}\rangle\leqslant s$.
\end{remark}

\begin{Lem} \label{ap3Crucial}
If $(\nu,s)\in A_k$ for $1 \leqslant k \leqslant d$ and  $1
\leqslant j < p(k)$ the following inequality holds
\begin{equation} \label{cla}
 \langle \nu, \l_{p(k)} - \l_j \rangle  \leqslant s- \langle \nu ,  e_{q_{\l_j}} \rangle.
\end{equation}
If in addition  $n(k) >0$ and ${m(k)} \ne 0$ we have
\begin{equation} \label{cla2}
\langle \nu ,  e_{q_{\l_j}} \rangle  \leqslant \langle \nu, e_{m
(k) } - \l_{n(k)} + \l _j \rangle.
\end{equation}
\end{Lem}
\begin{proof}
Suppose first that $\langle \nu, \l_j \rangle \geqslant
\langle \nu,    e_{q_{\l_j}} \rangle$ for some $1\leqslant j
\leqslant p(k)$.  It follows that (\ref{cla2}) holds by Remark
\ref{op2}. We deduce also that (\ref{cla}) holds since $\langle
\nu, \l_{p(k)} \rangle \leqslant s$ hence  $
   \langle \nu, \l_{p(k)} - \l_j \rangle  \leqslant  s - \langle \nu, \l_j \rangle \leqslant s -
 \langle \nu,    e_{q_{\l_j}} \rangle$.

We deal first with the case $1 \leqslant j < n(k)$. If   $\langle
\nu, \l_j \rangle < \langle \nu, e_{q_{\l_j}} \rangle$ then the
set $\{1\leqslant r\leqslant j\ |\ \l_r=\l_{n(k')}\mbox{ for
}1\leqslant k'\leqslant k\}$ is non empty and $n(k)>0$.  We
introduce the terms $\l_{r_0},\ldots,\l_{r_h}=\l_{n(k)}$, which
are the characteristic exponents appearing in the expansion of
the terms $w_{k_0},\ldots,w_{k_h}$ defined in Proposition
\ref{PruebaAlgQO} by the vector $\nu$ for suitable integers
$k_0,\ldots,k_h$.  First we set
\[ r_0:=\max \{1\leqslant r\leqslant j\ |\ \l_r=\l_{n(k')}\mbox{ for } 1\leqslant k'<k\}, \  k_0:=\min\{1\leqslant k'<k\ |\
\l_{r_0}=\l_{n(k')}\}, \] and then inductively
$
r_i:=\min\{r_{i-1}<  r\leqslant n(k) | \l_r=\l_{n(k')},
\mbox{
 } k_{i-1} < k'\leqslant k \}$ and $
k_i:=\min\{k_{i-1}\leqslant k'\leqslant k |
\l_{r_i}=\l_{n(k')}\}$. After finitely many steps we have an
integer  $h$ such that $k_h \leqslant k$ and  $r_h=n(k)$, and the
process stops.

It is easy to check that $r_0 \leqslant j < r_1$
$\langle\nu,\l_{r_0}\rangle\leqslant\langle\nu,\l_j\rangle<\langle\nu,\l_{r_1}\rangle$.
Notice that $e_{q_{\l_{r_l}}} = e_{m(k_l)} $ for $l = 0 , \dots,
h-1$ and    if ${m(k)} \ne 0$ we have also that $e_{q_{\l_{r_h}}}
= e_{m(k_h)} = e_{m(k)}  $.

Remark that the definition of
$w_{k_1},\ldots,w_{k_h}$ in Proposition \ref{PruebaAlgQO} involves
the choice $b_{k_i}$ in (\ref{rithm}), i.e., we have:
    \begin{equation}
    \phi_{k_{l-1}}(\nu)=\langle\nu,\l_{r_l}-\l_{r_{l-1}}+e_{q_{\l_{r_{l-1}}}}\rangle
<\langle\nu,e_{q_{\l_{r_l}}}\rangle\ \mbox{ for }l=1,\ldots,h-1.
    \label{ecCrucialc1}
    \end{equation}
and similarly by Lemma \ref{c1} and Definition \ref{Ak},
             \begin{equation}
    \phi_{k_{h}}(\nu)=\langle\nu,\l_{r_h}-\l_{r_{h-1}}+e_{q_{\l_{r_{h-1}}}}\rangle
\leqslant
 \left\{ \begin{array}{lcl}
 s    & \mbox{ if } &  {m(k)} = 0,
\\
 \langle \nu,  e_{m(k)}   \rangle    & \mbox{ if } &       {m(k)} \ne 0.
        \end{array}
\right.
    \label{ecu-h}
    \end{equation}

Summing  the inequalities in (\ref{ecCrucialc1}) for
$l=1,\ldots,h-1$ with (\ref{ecu-h}) provides the inequality
\begin{equation}
\langle\nu,\l_{r_h}-\l_{r_0}+e_{q_{\l_{r_0}}}\rangle  \leqslant
\left\{
\begin{array}{lcl}
 s    & \mbox{ if } &  {m(k)} = 0,
\\
 \langle \nu,  e_{m(k)}   \rangle    & \mbox{ if } &       {m(k)} \ne 0.
        \end{array}
\right. \label{ecCrucialc1prima}
\end{equation}

Now we distinguish two cases:

(i) If $\ell(\l_{r_0})=\ell(\l_j)$ then $q_{\l_{r_0}}=q_{\l_j}$ by
definition. Then the inequality (\ref{cla2}) hold
 by (\ref{ecCrucialc1prima}) since
$\langle\nu,\l_{r_0}\rangle\leqslant\langle\nu,\l_j\rangle$.
Adding (\ref{ecCrucialc1prima}) with the inequality
\begin{equation}    \label{ecu-p}
             \langle \nu, \l_{p(k)} -  \l_{n(k)} + e_{m(k) } \rangle \leqslant s
\end{equation}
shows that (\ref{cla}) holds.

(ii) Otherwise $\ell(\l_{r_0})\varsubsetneq\ell(\l_j)$ and we have
that $\l_j\neq\l_{n(k')}$ for $k_0<k'<k$.  The characteristic
exponent $\l_j$ is not chosen at any step of the algorithm  in
Proposition \ref{PruebaAlgQO}, that is,
\begin{equation}
\langle\nu,e_{q_{\l_j}}\rangle<\langle\nu,\l_j-\l_{r_0}+e_{q_{r_0}}\rangle.
\label{ecCrucialc1primaDos}
\end{equation}
Then we check that (\ref{cla2}) holds by adding the  inequalities
(\ref{ecCrucialc1primaDos}) and (\ref{ecCrucialc1prima}). The same
happens for (\ref{cla}) by adding (\ref{ecCrucialc1primaDos}),
(\ref{ecCrucialc1prima})     and (\ref{ecu-p}). It remains to
prove that (\ref{cla}) and  (\ref{cla2}) hold for $n (k) \leqslant
j < p(k)$. In this case $e_{q_{{\l}_j}} = e_{m(k)}$ and the
inequalities hold trivially by the definition of $p(k)$ since
$\langle \nu, \l_{n(k)} \rangle \leqslant  \langle \nu, \l_{j}
\rangle$.
 \end{proof}

\section{The jet space $j_s (H_{S,\nu}^*)$} \label{tor-jetQO}

\begin{Not}        \label{nota}
In this section  we fix $(\nu, s) \in A_k$ for some $0 \leqslant k
\leqslant d$ and we simplify our notations.
 Let $w_k = e_{j_1^{(k)}} + \dots
+ e_{j_k^{(k)}}$ be the vector defined by $\nu$ in Proposition
\ref{PruebaAlgQO}. We relabel the vectors in $\{e_1,\ldots,e_d\}$
in such a way that $w_k=e_1+\cdots+e_k$ if $n(k)=0$, and
$w_k=e_1+\cdots+e_{k-1}+\l_{n(k)}$ if $n(k) >0$. We denote the
integers $n(k), m(k)$ and $p(k)$, defined in the Section
\ref{conv-newtonQO} in terms of $\nu$, $k$ and $s$, simply by
${n}$, ${m}$,  and ${p}$, respectively.
\end{Not}

We begin by recalling  some definitions and results from
\cite{CoGP}. Let  $\{m_1, \dots, m_d\} $ be a basis of the lattice
$M$. We consider  a set of variables  $\{ c(m_1),\dots, c(m_d) \}
\cup \{ u_j(m_i) \}_{i=1,\dots, d}^{j\geqslant 1}$ to  define the
$\C$-algebra
\[ \mathcal{A}_{T_N} := \C[ c(m_1)^{\pm 1}, \dots, c(m_d)^{\pm
1}]\otimes_\C \C[ u_j(m_i)]_{i=1,\dots, d}^{j\geqslant 1}. \]
Since $\{m_1, \dots, m_d\} $ form a basis of $M$ there is a unique
homomorphism of semigroups
\begin{equation} \label{family}
M \rightarrow\mathcal A_{T_N}[[t]]^*
\end{equation} given by
$m_i\mapsto c(m_i)u(m_i)$ where $u(m_i)=1+\sum_{j\geqslant
1}u_j(m_i)t^j$ for $i=1,\ldots,d$. We associate to $w \in M$
the terms $c(w)$ and $u_j(w)$, for
$j\geqslant 1$, in the ring  $\mathcal A_{T_N}$, by  $w \mapsto
c(w) u(w)$ in (\ref{family}),  where $c(w)$ is the constant term
of the series and $u(w)$ is of the form $u(w) = 1+\sum_{j\geqslant
1}u_j(w)t^j$. The following result show some relations among the
elements $u_i ({w}) \in \mathcal A_{T_N}$, when we vary $i$ and
${w} \in M$, in terms of linear dependency relations among the
${w} \in M$.

\begin{Lem} \label{vectorQO}  {\rm (see Lemma 6.2 in \cite{CoGP}).}
Let ${w}_1, \dots, {w}_k$ be  linearly independent vectors in the
lattice $M$ spanning the linear subspace $\ell$ of $M_{\Q}$. For
any $w \in \ell$ and $s \geqslant 1$ the term $u_s ({w})$ belongs
to $ \Q [ u_1( {w}_j ), \dots, u_s ({w}_j ) ]_{j=1}^k.$.
\end{Lem}

We use in an essential manner the parametrization of the sets
$H^*_{S, \nu}$ by the arc space of the torus $T_N$. We have an
homomorphism of semigroups:
\begin{equation} \label{fam}
(\s^\vee \cap M, +) \rightarrow (\mathcal{A}_{T_N}[[t]], \cdot)
\mbox{ given by  } m   \mapsto t^{\langle\nu,m\rangle}c(m)u(m)
\mbox{, for } m\in\s^\vee\cap M.
\end{equation}
It defines a  parametrization of $H_{\bar{S}, \nu}^*$ by the arc
space of the torus $T_N$. Recall that by Lemma \ref{normal} we
have that the analytic algebra  of the germ $(S,0)$ is $\mathcal
O_S=\C\{\s^\vee \cap M_0 \} [\z]$. The maximal ideal of $\mathcal
O_S$ is $(x_1, \dots, x_{d+1})$  where $x_i:= X^{e_i}$ for
$i=1,\ldots,d$ and $x_{d+1}:=\z$ (see the notations of Section \ref{secQO}). The restriction of the homomorphism
$\C \{ \s^\vee \cap M \} \rightarrow \mathcal A_{T_N}[[t]]$ defined by
(\ref{fam}) to the local algebra  $\mathcal{O}_S$ parametrizes the
set $H_{S, \nu}$. This homomorphism verifies that
\begin{equation}
x_i \mapsto  t^{\langle \nu, e_i\rangle}  c(e_i) u (x_i),
\end{equation}
where $u(x_i)$ is a series of the form $u (x_i ) = 1 + \sum_{j
\geqslant 1} u_j (x_i)$. For $1 \leqslant i \leqslant d$ we have
that $c(e_i) u(x_i)$ is the image of $e_i$ by  the map
(\ref{family}), in particular $u(x_i) = u (e_i)$. We use the
expansion (\ref{expan}) of $x_{d+1} = \sum \b_\l X^\l$ as a power
series in $\C \{ \s^\vee \cap M \}$ to describe expansions of the
terms $u_r (x_{d+1})$ in terms of $u_l (e_i)$, for $i= 1,\dots,
d+g$ and $1\leqslant l \leqslant r$.

\begin{Not}
For $1\leqslant j\leqslant g$ we denote by $C_j$ the $\C$-algebra
of $\mathcal A_{T_N}$ generated by $c(m)$ for $m \in M_j$. We
denote $C_{g}$ simply by $C$. \label{defCj}
\end{Not}

\begin{Not}\label{notr_j}
If $\nu \in   \stackrel\circ\s\cap N$ and $ r \geqslant 1$ are
fixed
 we set the
sequence $r_1,\ldots,r_g$ by:
\begin{equation}
r_1 =r,\ r_2=r_1-\langle\nu,\l_2-\l_1\rangle,\ldots, r_g =r_1-\langle\nu,\l_g-\l_1\rangle.
\label{rsequence}
\end{equation}

\end{Not}

\begin{remark} \label{salva}
Notice that if $ r\leqslant s-\langle\nu,\l_1\rangle$ then we have
that
 $r_j\leqslant s-\langle\nu,\l_j\rangle$ for every $1\leqslant j\leqslant g$.
\end{remark}

\begin{Pro} \label{ur}
Let   $r\geqslant 1$  be an integer. In $\mathcal A_{T_N}$ we have the expansion:
\begin{equation}\label{zc2}
u_r(x_{d+1})=u_r(\l_1)+\sum_{l+\langle\nu,\l-\l_1\rangle=r}^{\l\geqslant\l_1}\theta(\l)u_l(\l),
\mbox{ with } \theta(\l):=\beta_\l
\beta_{\l_1}^{-1}c(\l)c(\l_1)^{-1}.
\end{equation}
 We use Notation \ref{notr_j}. If $1\leqslant j \leqslant g$ and  $r_j\geqslant 0$ then we
 set
\begin{equation}
\alpha_{r_j}^j
:=\theta(\l_j)u_{r_j}(\l_j)+\sum_{\langle\nu,\l-\l_j\rangle+l_j=r_j}^{\langle\nu,\l_j\rangle<\langle\nu,\l\rangle\leqslant\langle\nu,\l_{j+1}\rangle,\l\neq\l_{j+1}}\theta(\l)u_{l_j}(\l),
\label{alfa-r-j}
\end{equation}
The following properties hold:
\begin{enumerate}
\item[(i)]
 $u_r(x_{d+1})= \sum_{j=1, \dots, g}^{r_j \geqslant 0} \alpha_{r_j}^j$.
\item[(ii)] The coefficient $\theta(\l)$ of a term $u_l(\l)$
    in $\alpha_{r_j}^j$ belongs to $C_j$.
\item[(iii)] If the term $u_l(\l)$ appears in
    $\alpha_{r_j}^j$ then $1\leqslant l<r_j$ unless $\l=\l_j$
    and $l=r_j$, for $j=1,\ldots,g$.
\end{enumerate}
\label{Prop7-1}
\end{Pro}
\begin{proof}We deduce that
\begin{equation}
u(x_{d+1}) =
 \sum_{\l \geqslant \l_1}  \,  \beta_{\l}\beta_{\l_1}^{-1}  \,    {c}( \l)c(\l_1 )^{-1} \,    t^
{\langle \nu, \l - \l_1 \rangle} ( \sum_{l \geqslant 0} u_l (\l)
t^l ) . \label{exp(d+1)}
\end{equation}
by comparing the definition of $u(x_{d+1})$ with the expansion
(\ref{expan}) of $x_{d+1}$ in $\C \{ \s^\vee \cap M \}$. The
equality   follows by collecting the terms in $t^r$ in
(\ref{exp(d+1)}). The sum in (\ref{zc2}) is finite because there
is a finite number of lattice points $\l$ verifying that
$\l\geqslant\l_1$ and
$\langle\nu,\l\rangle\leqslant\langle\nu,\l_1\rangle+r$, since
$\nu\in\stackrel\circ\s\cap M$. Notice that $u_0(\l)=1$ by
definition and the term obtained for $l=0$ (resp. $l=r$) in the
sum (\ref{zc2}) is
$\sum_{\langle\nu,\l-\l_1\rangle=r}\beta_\l\beta_{\l_1}^{-1}c(\l)c(\l_1)^{-1}$
(resp. $u_r(\l_1)$).
 The other assertions  are obtained by collecting the indices in the sum (\ref{zc2}) according to
 (\ref{rsequence}). \end{proof}

\begin{remark}
For simplicity we convey
 that the term $\alpha_{r_j}^j$ equals zero whenever $r_j<0$.
 Hence we can write
 $u_r (x_{d+1} ) = \sum_{j =1}^g \alpha_{r_j}^j$ for any $r \geqslant 1$
 \end{remark}

\begin{Pro} \label{alg-indQO}  {\rm (see Proposition 6.3 in \cite{CoGP}).}
If the vectors $ {w}_1, \dots, {w}_k$ in $M$ are linearly
independent  then the terms, $ u_i({w}_1), \dots, u_i({w}_k)$ for
$ i \geqslant 1$, are algebraically independent over the field of
fractions of  $C$, and in addition if $k =d$ they  generate
$\mathcal{A}_{T_N}$ as a $C$-algebra.
\end{Pro}
\begin{Pro} \label{atn}
Let us denote by $B_0$ the $\C$-algebra \[ B_0 := \C [c(x_i), u_1
(x_i), u_2( x_i), u_3 (x_i), \dots ]_{i=1, \dots d+1}. \] We set
$F_j := \{ c( m ) \mid m \in \s^\vee\cap M_j \}$ for $j =1, \dots,
g$. Suppose that for $1 \leqslant j \leqslant g$ the algebras
$B_0, \dots, B_{j-1}$ have been defined by induction. Then we have
that $F_j$ is a multiplicative subset of $B_{j-1}$ and the
localization $B_j:= B_{j-1} [F_j^{-1}]$ is a subalgebra of
$\mathcal{A}_{T_N}$ which is equal to $\mathcal{A}_{T_N}$ if $j
=g$.
\end{Pro}
\begin{proof} Notice that $\C[F_j][F_j^{-1}]=C_j$. By definition we have
that $c(x_i) = c(e_i)$ for $i= 1,\dots, d+1$ hence it follows that
$F_1 \subset B_0$ is a multiplicative subset. Then we apply
Proposition \ref{ur} for $r = \langle \nu, \l_2 - \l_1 \rangle$.
We deduce that $u_r (x_{d+1}) = \a_{r_1} + \theta ( \l_2)$ where
$\a_{r_1}$ belongs to $B_1$. It follows that $c(\l_2) = \b_{\l_1}
\b_{\l_2}^{-1} c (\l_1) \theta (\l_2)$ belongs to $B_1$ and this
implies that $F_2$ is a multiplicative subset of $B_1$. Suppose
that the assertion is true for $1 \leqslant j-1 \leqslant g$. Then
we apply Proposition \ref{ur} for $r = \langle \nu, \l_j \rangle$
and we deduce similarly that $F_{j}$ is a multiplicative subset of
$B_{j-1}$. Finally, if $j =g$ we deduce that $B_g$ is a
$C$-subalgebra of $\mathcal{A}_{T_N}$ (see Notation \ref{defCj})
which contains the set $\{ u_1(e_i), u_2 (e_i), u_3 (e_i) \dots
\}_{i=1, \dots, d}$. This set generate
 $\mathcal{A}_{T_N}$ as a  $C$-algebra by
Proposition \ref{alg-indQO}, since the vectors $e_1, \dots, e_d$
are linearly independent.
\end{proof}


The idea of the proof of Theorem \ref{keyQO} is rather similar to
that of Proposition \ref{atn}, though we have to precise the form
of the terms $u_r (x_i)$ which remains when taking a $s$-jet. This
is controlled by the following Proposition:

\begin{Pro} \label{ur2}
Suppose that $(\nu, s) \in A_k$. We consider an integer $r$  such
that
\[
1 \leqslant r \leqslant \left\{\begin{array}{lcl} s -  \langle
\nu, \l_{1} \rangle  & \mbox{ if }  &  {n} = 0
\\
s -  \langle \nu, e_{{m} }  - \l_{{n}}  + \l_{1} \rangle    & \mbox{ if }  &
{n} > 0 \mbox{ and } {{m}} \ne  0
\\
s - \langle \nu, \l_{{n}} - \l_1 \rangle   & \mbox{ if }  &   {n} > 0 \mbox{
and } {{m}} =  0
\end{array}\right.
\]
The terms   $ \a_{j}^{r_{j}}$
in
 the expansion $u_r (x_{d+1}) = \sum_{j=1}^g
\a_j^{r_j}$ vanish for ${p} <  j \leqslant g$.  If in addition,
the term $u_l (\l)$ appears with non zero coefficient in the
expansion of $\a_i ^{r_i}$, for some $1 \leqslant i \leqslant
{p}$, then
\begin{equation} \label{ur3}
\l \in \ell_\nu^s \cap M_i \mbox{ and } l \leqslant s - \langle
\nu, e_{q_{\l}} \rangle.
\end{equation}
\end{Pro}
\begin{proof}By Proposition \ref{ur} we have that a term
$\a_j^{r_j}$ may be non-zero  in the expansion of $u_r(x_{d+1})$
if $r_j \geqslant 0$ (see Notation \ref{notr_j}). By Remark
\ref{salva} this happens if
\[
0 \, \leqslant \left\{
\begin{array}{lcl} s -  \langle \nu, \l_{j}
\rangle  & \mbox{ if }  &  {n} = 0,
\\
s -  \langle \nu, e_{{m} }  - \l_{{n}}   + \l_{j} \rangle   & \mbox{ if }  &
{n} > 0             \mbox{ and } {{m}} \ne 0,
\\

s -    \langle \nu, \l_{{n}} - \l_j \rangle   & \mbox{ if }  &  {n} > 0
\mbox{ and } {{m}} =  0.
\end{array}\right.
\]
This holds if and only if $j \leqslant {p}$ (see Notation
\ref{nota}).

Suppose that $u_l (\l)$ appears with non zero coefficient in
$\a_j^{r_j}$ for some $1 \leqslant j \leqslant {p}$.  This implies
that $X^\l$ appears in the expansion (\ref{expan}) of $\z$,
$\langle \nu , \l_j \rangle \leqslant \langle \nu, \l \rangle
\leqslant \langle \nu, \l_{j+1} \rangle$ and $\l \ne \l_{j+1}$,
thus $\l \in M_j$ and $\l \leqslant_\nu \l_j$. Then we obtain
\[
0 \leqslant r_j = \langle \nu, \l - \l_j \rangle + l  \leqslant
\left\{
\begin{array}{lcl} s -  \langle \nu, \l_{j}
\rangle  & \mbox{ if }  &  {n} = 0,
\\
s -  \langle \nu, e_{{m} }  - \l_{{n}}   + \l_{j} \rangle    & \mbox{ if }  &
{n} > 0      \mbox{ and } {{m}} \ne 0 ,
\\
 s -    \langle \nu, \l_{{n}} - \l_j \rangle     & \mbox{ if }  &  {n} > 0
\mbox{ and } {{m}} = 0.
\end{array} \right.
\]
We deduce from these inequalities, using  (\ref{cla}), $\l_j
\leqslant_\nu \l$ and $\l_{{n}} \leqslant_{\nu} \l_{{p}}$ that
\begin{equation}        \label{cla3}
 l \leqslant   s - \langle \nu, e_{q_{{\l}_j}} \rangle.
\end{equation}

Since $\l \in M_j$ is $\geqslant_{\nu} \l_j$ we have a unique
expansion of the form $\l = \l' + \sum_{e_i \notin \ell (\l_j)}^{
1\leqslant i \leqslant d}  h_i e_i$, where $\l' \geqslant \l_j$
belongs to $\ell (\l_j)$ and $h_i \geqslant 0$ are integers. Since
$\langle \nu, \l \rangle \leqslant s$ it follows from Proposition
\ref{crucial} that $\l $ belongs to $\ell_{\nu}^s$. We deduce that
$\langle \nu, \l \rangle \geqslant \langle \nu, \l'\rangle $ and
$\langle \nu , \l \rangle \geqslant \langle \nu, e_i \rangle$ if
$h_i > 0$. This implies that
 $l \leqslant s - \langle \nu, e_i \rangle$ if $h_i \ne 0$ and together with (\ref{cla3})
proves that $l \leqslant s - \langle \nu, e_{q_{\l}} \rangle$.
\end{proof}

\begin{Not}
We denote by $\bar{B}_{0}$ the $\C$-algebra of $\mathcal A_{T_N}$
generated by:
\begin{equation} \label{generQO}
c(x_i), u_1 (x_i), \dots , u_{s - \langle \nu, e_i \rangle} (x_i)
\mbox{ for those } 1\leqslant i\leqslant d+1  \mbox{ such that }
\langle \nu, e_i \rangle \leqslant s.
\end{equation}
 We set $\bar{C}$ the $\C$-algebra generated by $\{ c(e_j) \mid \langle\nu,e_j\rangle\leqslant s, 1\leqslant j\leqslant d+p \}$
\label{algebraO}
\end{Not}

\begin{Pro}    \label{pro-aux}
We use Notations \ref{nota} and \ref{algebraO}. The
$\bar{C}$-algebra generated by  $\bar{B}_{0}$ coincides with the
$\bar{C}$-algebra of $\mathcal{A}_{T_N}$ generated by the
following elements:
\begin{equation}   \label{one}
               \{  u_1 (e_i), \dots, u_{s-\langle\nu,e_i\rangle} (e_i)  \}_{i = 1, \dots, k-1}
\end{equation}
        together with
\begin{equation}     \label{two}
\left\{
\begin{array}{lcl}
\{   u_l (e_k) \} _{l=1}^{s - \langle \nu, e_k \rangle}
 &  \mbox{ if }  &  {n} =0
\\
\{  u_l (e_{{m}}) \} _{l = 1}^{s - \langle \nu, e_i \rangle} \cup
\{ u_{l} (x_{d+1}) \}_{l =s -  \langle \nu, e_{{m}} - \l_{{n}} +
\l_1  \rangle + 1 }^{s - \langle \nu, \l_1 \rangle}
 &  \mbox{ if }    &    {n} >1               \mbox { and } {m}  \ne  0,
\\
 \{
u_{l} (x_{d+1}) \}_{l =  \langle \nu, \l_{{n}} \rangle + 1 }^{s -
\langle \nu, \l_1 \rangle}
 &  \mbox{ if } &  {n} > 0 \mbox{ and }  {m} = 0.
\end{array} \right.
\end{equation}
\end{Pro}
\begin{proof} We consider the case ${n} =0$ (resp. ${n} > 0$ and
${m} = 0$). By Lemma \ref{crucial} we have that if $\langle \nu,
e_i \rangle \leqslant s$ and $1\leqslant i \leqslant d$ then $1
\leqslant i \leqslant k$. (resp. $1 \leqslant i \leqslant k-1$).
It is sufficient to prove that the elements $u_r (x_{d+1}) $ for
$1 \leqslant r \leqslant s - \langle \nu, \l_1 \rangle$ belong to
the $\bar{C}$-algebra generated by (\ref{one}) and (\ref{two}). By
Propositions \ref{ur} and \ref{ur2} we get $u_r (x_{d+1} ) = \sum
\theta (\l) u_l (\l)$ where  $\l \in \ell_{\nu}^s \cap M_{{p}} $,
$\theta (\l) \in \bar{C}$ and $l \leqslant s - \langle \nu,
e_{q_\l} \rangle$. This implies that $\l \in \mathsf{span}_\Q
(e_1, \dots, e_k )$ (resp. $\l \in \mathsf{span}_\Q (e_1, \dots,
e_{k-1})$) and $l \leqslant s - \langle \nu, e_i \rangle$ for $i=
1, \dots, k$ (resp. for $i= 1, \dots, k-1$) thus by Lemma
\ref{vectorQO} we deduce that $ \theta (\l) u_l (\l) \in \bar{C} [
u_1 (e_i), \dots, u_{s - \langle \nu, e_i \rangle } (e_i) ]_{i=
1\ldots k (\mbox{{\rm \small resp. }}  k-1)}$.

We deal then with the case  ${n} > 0$ and ${m} \ne 0$. It is
sufficient to prove that the elements $u_r (x_{d+1}) $ for $1
\leqslant r \leqslant s -  \langle \nu, e_{{m}} - \l_{{n}} + \l_1
\rangle $ belong to the $\bar{C}$-algebra generated by (\ref{one})
and (\ref{two}). The conclusion follows by the same arguments, by
using Propositions       \ref{ur},  \ref{ur2} and Lemma
\ref{vectorQO}. \end{proof}

\begin{Pro} \label{atn2}
We set $\bar{F}_j:= \{ c(e_i) \mid \langle\nu,e_i\rangle\leqslant
s, 1\leqslant i\leqslant d+j \}$ for $j =1, \dots, {p}$. Suppose
that for some $1 \leqslant j < {p}$ the algebras $\bar{B}_{0},
\dots, \bar{B}_{j-1}$ have been defined by induction. Then
$\bar{F}_j$ is a multiplicative subset of $\bar{B}_{j-1}$ and the
localization $\bar{B}_j := \bar{B}_{j-1} [(\bar{F}_{j})^{-1}]$ is
a subalgebra of $\mathcal{A}_{T_N}$. Then the ring $ \bar{B} :=
\bar{B}_{{p}}$ is generated by $\bar{B}_{0}$ as a
$\bar{C}$-algebra. The ring $ \bar{B} $ is the coordinate ring of $j_s (H^*_\nu)$.
\end{Pro}
{\em Proof}. Set  $\bar{C}_j :=  \C [\bar{F}_j] [ (\bar{F}_j)^{-1}
]$.
 By definition we have that $c(x_i) = c(e_i)$ for $i=
1,\dots, d+1$, hence it follows that $\bar{F}_1 \subset \bar{B}_0$
is a multiplicative subset. We apply Propositions \ref{ur} and
\ref{ur2} for $ r=  \langle \nu, \l_2 - \l_1 \rangle$. We deduce
that $u_r (x_{d+1}) = \a_{r_1} + \theta ( \l_2)$ where $\a_{r_1}$
belongs to $\bar{B}_1$. It follows that $c(\l_2) = \b_{\l_1}
\b_{\l_2}^{-1} c (\l_1) \theta (\l_2)$ belongs to $\bar{B}_1$ and
this implies that $\bar{F}_2$ is a multiplicative subset of
$\bar{B}_1$. Suppose that the assertion is true for $1 \leqslant
j-1 \leqslant g$. Then we apply Proposition \ref{ur2} for $r =
\langle \nu, \l_j \rangle$ and we deduce similarly that
$\bar{F}_j$ is a multiplicative subset of $\bar{B}_{j-1}$. We can
iterate this procedure for $j= 1, \dots, {p}$. By Proposition
\ref{pro-aux} it follows that $\bar{B}$ is the  $\bar{C}$-algebra
generated by $\bar{B}_0$. By construction $\bar{B}$ is the 
coordinate ring of $j_s (H^*_\nu)$ (compare with Proposition
\ref{atn}). \hfill $\Box$

\begin{Lem}
The generators  (\ref{one}) and (\ref{two}) of the
$\bar{C}$-algebra $\bar{B}$  are algebraically independent over
the field of fractions of $ \bar{C}$. \label{Lemaprekey}
\end{Lem}
\begin{proof}If ${n} >0$ and if $\langle \nu , \l_{{n}} \rangle <
r$ then we have that $u_r (x_{d+1})$ is of the form:
\begin{equation} \label{flu2}
 u_{r} (x_{d+1} ) = \theta_{\l_{{n}}} u_{r_{{n}}} (\l_{{n}} )  +
\sum_{\l  \ne \l_{{n}}} ^{ l \leqslant   r_{{n}}} \theta_\l  u_l
(\l) + \sum_{ e_1 \y \cdots \y e_{k-1} \y \l = 0} \theta_\l
u_{r_{{n}}} (\l).
\end{equation}
(see Notation \ref{notr_j}).
By Proposition \ref{ur} we have that $u_r (x_{d+1})  = \sum_{j=1}^g \a^j_{r_j}$ and if
 $u_l( \l)$ appears in the expansion (\ref{alfa-r-j}) of $\a^j_{r_j} $ with non zero coefficient
then  $l = r - \langle \nu, \l - \l_1 \rangle$
and  $\l \in M_j$ verifies that
\begin{equation}   \label{flu}
\langle \nu, \l_{j} \rangle <  \langle \nu, \l \rangle  \leqslant
\langle \nu, \l_{j+1} \rangle.
\end{equation}
If $1 \leqslant j < {n}$ then  we have the inequality $ l = r -
\langle \nu, \l - \l_1 \rangle  \leqslant  r - \langle \nu,
\l_{{n}} - \l_1 \rangle = r_{{n}}$.  The equality $l = r_{{n}}$
implies that $\langle \nu, \l \rangle = \langle \nu, \l_{{n}}
\rangle$, thus
 $\l$ belongs to $M_{{n} -1} \cap \mbox{{\rm span}}_\Q (e_1, \dots, e_{k-1})$
since $\langle \nu, \l_{{n}} \rangle < \langle \nu, e_{{m}}
\rangle$. If $ {n} \leqslant j  $  then we have that $l < r_n$ by
(\ref{flu}).

Now we prove the Proposition by distinguishing the following cases:
\begin{enumerate}
\item[(i)] If ${n} =0$ then the assertion is consequence of
Proposition \ref{alg-indQO}.

\item[(ii)] If  ${n} > 0$ and ${m} = 0$ the assertion is
consequence of (\ref{flu2}) and Proposition \ref{alg-indQO}.

\item[(iii)] If    ${n} > 0$ and ${m} \ne 0$ then by Proposition
\ref{vectorQO} and the definition of $e_{{m}}$ we have that the
$\bar{C}$ algebra generated by $
 \{ u_l (e_1), \dots, u_l (e_{k-1})
 \} _{l = 1}^{s - \langle \nu, e_i \rangle}  \cup  \{  u_l (e_{{m}}) \}_{l = 1}^{s - \langle \nu, e_{{m}} \rangle},
$ coincides with the   $\bar{C}$ algebra   generated by
$
 \{ u_l (e_i)
 \} _{l = 1, \dots, s - \langle \nu, e_i \rangle} ^{i=1, \dots, k-1}  \cup  \{  u_l ( \l_{{n}} ) \}_{l = 1}^{s - \langle \nu, e_{{m}} \rangle}
 $.
We deduce the statement from this, the equation (\ref{flu2}) and
Proposition \ref{alg-indQO}.
\end{enumerate}
\end{proof}

\medskip

{\em Proof of Theorem \ref{keyQO}.} The statement is clear if
$(\nu,s) \in A_0$. Suppose that $(\nu,s) \in A_k$ for some
$1\leqslant k\leqslant d$ and $s >0$. By Lemma \ref{Lemaprekey} we
have $sk- \ord_{\J_k}(\nu)$ generators of $\bar{B}$ as a $\bar{C}$
algebra. These generators, indicated in  (\ref{one}) and
(\ref{two}), are algebraically independent over the field of
fractions of $\bar{C}$. Notice that $\bar{C}$ is the $\C$ algebra
of the rank $k$ sublattice of $\ell_\nu^s\cap M_{{p}}$ hence
$\Spec \bar{C}$ is a $k$-dimensional torus. It follows from
this and Proposition \ref{atn2} that we can define an isomorphic
parametrization
\begin{equation}    \label{etapar}
 (\C^*)^k\times\A_\C^{sk-\mathsf{
ord}_{\J_k}(\nu)} \rightarrow \Spec\bar{B} \cong j^s
(H^*_\nu).
\end{equation}
\mathproofbox

\section{Description of the series $P(S)$ and proofs of the rationality results}
\label{GenptOrbitsQO}  First we study the
relations between the sets $j_s (H^*_\nu)$ when $\nu$ varies.
\begin{definition} \label{eqQO}
Define an equivalence relation in the set ${A}_k$ for $1 \leqslant
k \leqslant d$:
\begin{equation} \label{equivalence-bis}
(\nu,s) \sim (\nu',s) \in {A}_k \Leftrightarrow \left\{
\begin{array}{c}
s=s',\ \nu  \mbox{ and  }  \nu' \mbox{ define the same face of }
{\mathcal N}( \mathcal{J}_j)
\\
\mbox{ and } {\mathsf{ ord}}_{{\mathcal J}_j} (\nu) = {\mathsf{
ord}}_{{\mathcal J}_j} (\nu') \mbox{ for }1\leqslant j\leqslant k.
\end{array}
 \right.
\end{equation}

We denote by $[(\nu,s)]$ the equivalence class of $(\nu,s)\in A_k$ by this relation.
\end{definition}

\begin{remark} \label{remkQO}
\
\begin{enumerate}
\item [(i)] The set $\{[(\nu,s_0)]\ |\ (\nu,s_0)\in {A}_k\}$ is finite
    for $s_0>0$ and $1\leqslant k\leqslant d$.

\item [(ii)] If $k=d$ the equivalence relation is the
    equality.

\end{enumerate}
\end{remark}

\begin{Pro}    {\rm (cf.~Prop. 8.3 \cite{CoGP})}
Let $(\nu,s),(\nu',s)\in A_k$. The following are
equivalent:
\begin{enumerate}
\item [(i)] $(\nu,s)\sim(\nu',s).$
\item [(ii)] $\ell_\nu^s =\ell_{\nu'}^s $ and $\nu_{|
    \ell^s_\nu }=\nu'_{| \ell^s_{\nu'} }.$
\item [(iii)] $j_s(H_\nu^*)=j_s(H_{\nu'}^*)$.
\item [(iv)] $j_s(H_\nu^*)\cap j_s(H_{\nu'}^*)\neq\emptyset$.
\end{enumerate}
\label{Prop-equivQO}
\end{Pro}
\begin{proof} The equivalence between (i) and (ii) follows by
Lemma \ref{crucial}, Definition \ref{Ak} and the definition of the
equivalence relation.

If (ii) holds, then the vectors obtained in the algorithm in
Proposition \ref{PruebaAlgQO} coincides for both $\nu$ and $\nu'$.
It follows that the indices $m(k)$, $n(k)$ and $p(k)$ are the same
for $\nu$ and $\nu'$. These facts imply that the isomorphic
parametrizations (\ref{etapar}) corresponding to
 $(\nu,s),(\nu',s) \in A_k$ coincide. Hence $j_s(H_\nu^*)=j_s(H_{\nu'}^*)$ and (iii) holds.

If (iv) holds we prove that (ii) holds. There exists $h\in
H_\nu^*$ and $h'\in H_{\nu'}^*$ such that $j_s(h)=j_s(h')$. Then
for any $i\in\{1,\ldots,d+1\}$ the inequality
$\langle\nu,e_i\rangle\leqslant s$ implies that
$\langle\nu,e_i\rangle=\langle\nu',e_i\rangle$. We denote by
$n(k)$ and $m(k)$ (resp. $n'(k)$ and $m'(k)$)
 the integers associated to $(\nu,s)$ (resp. $(\nu',s)$) in Proposition
 \ref{PruebaAlgQO}. If $n(k) =0$ then the assertion follows.
If $n(k) >0$ we set $\ell=\mbox{{\rm span}}_\Q \{ e_i \mid
1\leqslant i \leqslant d, \langle \nu, e_i \rangle \leqslant s \}$
and $x_{d+1}'=\sum_{\l_{n(k)}\nleqslant\l,\l\in\ell}\beta_\l
X^\l$. We have that $\nu_{| \ell}=\nu'_{| \ell}$. If $\beta_\l\neq
0$ appears in $x_{d+1}'$ and $\l=\sum_{i=1}^d a_i e_i$ then
$\langle\nu,\l\rangle\geqslant\langle\nu,e_i\rangle$ whenever
$a_i\neq 0$ (this is consequence of Lemma \ref{crucial}). In
addition, if $\langle\nu,\l\rangle\leqslant s$ then it follows by
Lemma \ref{vectorQO} that
$u_r(\l)\in\Q[u_1(e_i),\ldots,u_{s-\langle\nu,e_i\rangle}(e_i)]_{1\leqslant
i\leqslant d}^{\langle\nu,e_i\rangle\leqslant s}$ for
$r=1,\ldots,s-\langle\nu,\l\rangle$. Then it is easy to see that
the initial coefficients of $X^{\l_j}\circ h$ and $X^{\l_j}\circ
h'$ coincide for $i=1,\ldots,n(k)-1$. We deduce that
$j_s(x_{d+1}'\circ h)=j_s(x_{d+1}'\circ h')$ hence
$j_s ( (x_{d+1}  -  x_{d+1}') \circ h)
$ is of order $\langle \nu, \l_{n(k)} \rangle$. Since
$j_s ( (x_{d+1}  - x_{d+1}') \circ h) = j_s ( (x_{d+1}  -
x_{d+1}') \circ h')$ we deduce that $\langle \nu, \l_{n(k)}
\rangle = \langle \nu', \l_{n(k)} \rangle$ hence $\l_{n(k)} =
\l_{n'(k)}$ and (ii) holds.
\end{proof}

\begin{Not} \label{auxiQO}
The cone $\hat{\s}: = \s \times \R_{\geqslant 0}$ is rational for
the lattice $\hat{N}: = N \times \Z$.
\begin{enumerate}
\item[(i)]  If $\t \subset \s$ and $1 \leqslant k \leqslant d $ we
set \[\t(k) := \{ (\nu, s) \mid \nu \in \stackrel{\circ}{\t} \cap
\stackrel{\circ}{\s} \mbox{ and } \f_k( \nu) \leqslant s <
\f_{k+1} (\nu) \}.\]

\item[(ii)] If $\t \in \cap _{i=1}^k \Sigma_i$ then we set $A_{k,
\t}: = \t(k) \cap \hat{N}$.
\end{enumerate}
\end{Not}

\begin{remark} \label{pwlQO}
If $\t$ is a cone contained in a cone of the fan $\cap_{i=1}^k
\Sigma_i$ then we have that if  $\t(k) \ne \emptyset$, then the
closure of $\t(k)$ in $\hat{\s}$ is  a convex polyhedral cone,
rational for the lattice $\hat{N}$ (since in this case the
functions $\ord_{\J_1}, \dots. \ord_{\J_k}$ , hence also $\f_1,
\dots, \f_k$ , are linear on $\t$  and the function
$\ord_{\J_{k+1}}$, hence also $\f_{k+1}$, is piece-wise linear and
convex on $\t$). The set $A_{k, \t}$ may be empty, for instance,
if $\t$ is contained in the boundary of $\s$ or if for all $\nu$
in the interior of $\t$ we have that $\f_k(\nu) = \f_{k+1} (\nu)$.
\end{remark}

\begin{remark} \label{86QO}
$\,$
\begin{enumerate}
\item[(i)] For $1 \leqslant k \leqslant d$ we have that  $A_{k} =
{\bigsqcup_{\t \in \cap _{i=1}^k \Sigma_i}} A_{k, \t}$.

\item[(ii)] The vectors  $(\nu,s), (\nu',s) \in A_k$ are
equivalent by relation $\sim$ in (\ref{equivalence-bis}) if and
only if there exists a cone $\t \in  \cap _{i=1}^k \Sigma_i$ such
that $\nu$ and $\nu'$ belong to the relative interior of $\t$ and
we have that $\f_i (\nu) = \f_i (\nu')$, for $i= 1, \dots, k$.

\item[(iii)]  It follows that $A_k /_{\sim} = {\bigsqcup_{\t \in
\cap _{i=1}^k \Sigma_i}} A_{k, \t} /_{\sim}$, where  $A_{k, \t}
/_{\sim}$ is the set of equivalent classes of elements in the set
$A_{k, \t}$ by relation (\ref{equivalence-bis}).
\end{enumerate}
\end{remark}

\begin{Pro}      {\rm (cf.~Prop. 8.7 \cite{CoGP})}
If $\nu \in \stackrel{\circ}{\s}$, $s \geqslant 1$  and $\theta
\leqslant \s$ then the following relations are equivalent:
\begin{enumerate}

 \item[(i)] $j_s(H_{\nu}^*)\cap j_s(H_{S_\theta}) \ne \emptyset$,

\item[(ii)] $j_s(H_{ \nu}^*)\subset
j_s(H_{S_\theta})$,

\item[(iii)] $\ell_\nu^s \subset\theta^{\perp}$.
\end{enumerate}
\label{thetapropositionQO}
\end{Pro}

\begin{proof}If (i) holds, there is an arc $h\in H_\nu^*$ with
$j_s (h) \in j_s(H_{S_\theta})$. Then if $j_s(X^{e_i}\circ h) \ne
0 $ the vector $e_i$ belongs to $\theta^\perp$. By Lemma
\ref{crucial} those vectors generate the subspace $\ell_\nu^s$ and
hence (iii) holds.

Suppose that (iii) holds. The vector $\nu':=\nu_{|\theta^\perp}$
belongs to the lattice  $N(\theta,\z)$ and it is in the interior
of the cone $\s/\theta\R$ (see Notation \ref{notlattQO}). By
definition we have that $\ell_\nu^s = \ell_{\nu'}^s$ and $\nu_{|
\ell_{\nu}^s} = \nu'_{| \ell_{\nu'}^s}$. Then we continue as in
the proof of (iv) $\Rightarrow$ (ii) in Proposition
\ref{Prop-equivQO}. \end{proof}

\begin{Pro} \label{newQO}
If $1\leqslant k \leqslant d$ and $(\nu, s) \in A_k$ then the
following assertions are equivalent:
\begin{enumerate}
\item[(i)] The intersection $j_s(H^*_{\nu}) \, \cap \, (
\bigcup_{0 \ne \theta \leqslant \s} j_s (H_{S_\theta}))$ is empty.

\item[(ii)]  The face ${\mathcal F}_\nu$ of the polyhedron $\mathcal{N}
(\J_k) $  determined by $\nu$ is contained in the interior of  $\s^\vee$.
\end{enumerate}
\end{Pro}
\noindent{\em  Proof.}  The proof coincides with that of
Proposition 8.8 \cite{CoGP}. \hfill $\ {\Box}$

\begin{definition} \label{SkQO}       {\rm (cf.~Def. 8.9 \cite{CoGP})}
If $1 \leqslant k \leqslant d$ we define the set  $\mathcal{D}  _k
$ as the subset of cones $\t \in \bigcap_{i= 1}^k \Sigma_i$ such
that the face $\mathcal{F}_\t$ of $\mathcal{N} (\J_k)$ is
contained in the interior of $\s^\vee$.
\end{definition}

\begin{remark} \label{ddQO}
Notice that $\mathcal{D}_d  =\bigcap_{i= 1}^d \Sigma_i$. If $\t
\in \mathcal{D}_d$, the set $\t(d)$ is non-empty if and only if
$\stackrel{\circ}{\t}  \subset \stackrel{\circ}{\s}$.
\end{remark}

As a consequence of the results of this Section we have the
following Propositions:
\begin{Pro} \label{29QO}          {\rm (cf.~Prop. 8.11 \cite{CoGP})}
Let us fix an integer $s_0  \geqslant 1$.
   We have the following partition as union of locally  closed subsets:
\begin{equation} \label{trunc-sQO}
j_{s_0}(H_S^*)\setminus\displaystyle\bigcup_{ 0
\neq\theta\leqslant \s}
j_{s_0}(H_{S_\theta})=\displaystyle\bigsqcup_{k=1}^d
\bigsqcup_{\t\in\mathcal{D}_k} \bigsqcup_{[ (\nu, s_0) ] \in A_{k,
\t} /_{\sim}}  j_{s_0} (H_{\nu}^*).
\end{equation}
\end{Pro}

If $s_0 \geqslant 1$ the coefficient of $T^{s_0}$ in the auxiliary
series $P(S)$ is obtained by taking classes in the Grothendieck
ring in (\ref{trunc-sQO}), and then using Theorem \ref{keyQO}.

\begin{Pro}              {\rm (cf.~Prop. 8.12 \cite{CoGP})}
\label{mk} If $\t \in \mathcal{D}  _k $ we set
\[ P_{k,\t}(S)=
(\L-1)^k \sum_{s\geqslant 1}   \sum_{[ (\nu, s_0) ] \in A_{k, \t}
/_{\sim}} \L^{sk-\ord_{\J _k}(\nu)}T^s.\]
  We have that $ P(S) =
\sum_{k=1}^d \sum_{\t \in \mathcal{D}_k} P_{k,\t}(S). $
\end{Pro}

The main results on the geometric motivic Poincar\'e series are
based on Theorem \ref{keyQO} and Proposition \ref{29QO}.
The proofs of the main results of this Section
follow by applying the method introduced in the toric case \cite{CoGP}.

Recall that if $\r\in\cap\Sigma_j$ is of dimension one we
denote by $\nu_\r$ the integral primitive vector in $N$.

\begin{Pro}  {\rm (cf.~Prop. 9.5 \cite{CoGP})} 
 If $1\leqslant k\leqslant d-1$ then the rational form of the
series $P_{k,\t}(S)$ is of the form:
\begin{equation} \label{P-jQO}
 \frac{  Q_{S, k, \t}}{\prod_{\r \leqslant
\t}^{\dim \r =1 }  ( 1 - \L^{\Psi_k (\nu_\r)} T^{\phi_k
(\nu_\r)})\prod_{\r\in\Sigma_{k+1},\r\subset\t, }^{\dim\r=1,
\phi_{k+1} (\nu_\r) \ne \phi_{k} (
\nu_\r)}(1-\L^{\Psi_{k+1}(\nu_\r)}T^{\phi_{k+1}(\nu_\r)})},
\end{equation}
for some $Q_{S,k, \t} \in \Z[\L, T]$. If $k=d$ then equation
(\ref{P-jQO}) holds by replacing in the denominator the term
$\prod_{\r\in\Sigma_{k+1},\r\subset\t}^{\dim\r=1}(1-\L^{\Psi_{k+1}(\nu_\r)}T^{\phi_{k+1}(\nu_\r)})$
by $1-\L^dT$. Both numerator and denominator of (\ref{P-jQO}) are
determined by the lattice $M$ and the Newton polyhedra of the
logarithmic jacobian ideals.
 \label{ratP_jQO}
\end{Pro}
\begin{proof} The proof follows by using the same argument of Proposition 9.5 \cite{CoGP}.
     \end{proof}

\begin{remark}
The factor $1-\L^dT$ does not appear in the denominator of
$P_{k,\t}(S)$ for $1\leqslant k\leqslant d-1$.
\end{remark}

\medskip

{\em Proof of Theorem  \ref{PLambdaRacQO}}. It is consequence of
Propositions \ref{ratP_jQO} and  \ref{mk}. \mathproofbox

\medskip

{\em Proof of Corollary         \ref{P-geomQO}}. It follows from Proposition \ref{descompPgeomQO} and Theorem
\ref{PLambdaRacQO}. \mathproofbox

       \medskip

{\em Proof of Corollary  \ref{vol-mot2QO}}.  The proof is the same
as Proposition 10.1  \cite{CoGP}. \mathproofbox

\section{Geometrical definition of the logarithmic jacobian ideals}
\label{torQO}

We introduce a distinguished sequence of monomial ideals of the
coordinate ring of the normalization of the q.o.~hypersurface.

Let $(Y,0)$ be a germ of complex analytic variety. Its analytic algebra $\mathcal{O}_Y$ is of the form
$\mathcal{O}_Y  = \C \{ X \}  / I$,  for $X = (X_1, \dots, X_n)$.  We denote
by  $\Omega_Y^1$ the $A$-module  of {\em   (K\" ahler) holomorphic differential forms}
and by $d :\mathcal{O}_Y  \rightarrow  \Omega_Y ^1$ its canonical
derivation.
We denote by $ \Omega^k_Y$ the $A$-module
$ \Omega^k_Y := \bigwedge^k \Omega^1_Y$. See Section 1.10 of   \cite{GLS} for instance.

First, we review the normal toric case following \cite{Oda}
Chapter 3 and \cite{LR} Appendix. We consider the toric
singularity $Z$ with analytic algebra  of the form $\mathcal{O}_{Z}= \C
\{\s^\vee \cap M \} $. We denote by $D$ the equivariant Weil
divisor defined by the sum of orbit closures of codimension one
in the toric variety $Z$. The $\mathcal{O}_{Z}$-module
$\Omega_{Z}^1 (\log D) $
 of $1$-forms of $Z$ with logarithmic poles along $D$ is
identified with $ \mathcal{O}_Z \otimes_\Z M $. We have a map
of $\mathcal{O}_Z$-modules $  \omega: \Omega_{Z}^1 \rightarrow
\mathcal{O}_Z \otimes_\Z M $, determined  by $d X^{\g} \mapsto
X^\g \otimes \g$, for $\g \in \s^\vee \cap M$. Notice that  if
$\{ \g_i \}_{i=1}^r $ generate the semigroup $ \s^\vee \cap M$
then $ \{  d X^{\g_i} \}_{i=1}^r$   generate the
$\mathcal{O}_Z$-module
 $\Omega_{Z}^1$.
If $\psi = \sum_{\g \in \s^\vee \cap M} c_\g X^\g $ then $d
\psi \mapsto        \w (d \psi) =  \sum_{\g \in \s^\vee \cap M}
c_\g X^\g \otimes  \g$. Notice that  $  \w (d \psi) =
 \sum_{i=1}^d ( \sum_{\g \in   \s^\vee \cap M}     c_\g \, \g_i X^\g )   \otimes u_i$, where
$(\g_1, \dots, \g_d) $ denote the coordinates of $\g$ in terms of
a basis $u_1, \dots, u_d$ of the lattice $M$. For $k =1, \dots, d$
we have the following homomorphism of $\mathcal{O}_{Z}$-modules
$\y^k \w :\Omega_{Z}^k \longrightarrow     \Omega_{Z}^k (\log D) =
\mathcal{O}_Z \otimes_\Z \bigwedge^k  M$
\begin{equation} \label{w}
 d X^{\g_1} \y \cdots \y    d X^{\g_k} \mapsto  X^{\g_1 + \cdots
+ \g_k } \otimes  \g_1 \y \cdots \y \g_k.
\end{equation}
Notice that fixing a basis $u_1,  \dots,  u_d$ of the rank $d$
lattice $M$ defines an  isomorphism $ \bigwedge^d M  \rightarrow
\Z$, given by $u_1 \y \cdots  \y u_d \mapsto 1$. This provides
an homomorphism of $\mathcal{O}_Z$-modules
\begin{equation}               \label{f}
\f: \Omega_{Z}^d (\log D)  \rightarrow \mathcal{O}_Z.
\end{equation}
The
image of $\Omega_Z^d$ by the composite
$\Omega_{Z}^d \rightarrow     \Omega_{Z}^d (\log D)   \rightarrow \mathcal{O}_Z $
is an ideal of   $\mathcal{O}_Z$, which is
independent of the basis of $M$ chosen. This ideal is called
the {\em logarithmic jacobian ideal} of $Z$ in \cite{LR}.

If  $(S, 0)$ is the germ of q.o.~singularity then its
normalization is a toric singularity and we have a canonical map $
 \eta: \Omega_{S}^k \rightarrow  \Omega_{\bar{S}}^k$ for $1\leqslant k\leqslant d$,
which induced by the normalization map $\bar{S} \rightarrow S$. We
denote also by $\f$  and $\y^k \w$ the maps (\ref{f})  and
(\ref{w}) if $Z = \bar{S}$.

\begin{definition}          {\rm (cf.~Def. 11.1 \cite{CoGP})}
The {\em $k^{th}$-logarithmic jacobian ideal} of $(S,0)$ is the
$\mathcal {O}_{\bar{S}}$-module generated by the set
$\phi(\wedge^k \omega (\y^k \eta (
\Omega_S^k)))\wedge\bigwedge^{d-k}M) \subset
\mathcal{O}_{\bar{S}}$.
\end{definition}

\begin{Pro} \label{jotak}  {\rm  (generalizing Theorem 3.3 of \cite{PedroNash},
  cf.~Prop. 11.2 \cite{CoGP})}
For $k=1,\ldots,d$ the $k^{th}$-logarithmic jacobian ideal of a
q.o.~hypersurface $(S,0)$ is the monomial ideal $\J_k$ of $\C \{
\s^\vee \cap {M} \} $ of Definition \ref{idealesJac}.
\end{Pro}

\begin{proof} We denote by $\J_k'$ the ideal generated by
$\phi(\wedge^k \omega (\y^k \eta (
\Omega_S^k)))\wedge\bigwedge^{d-k}M) \subset
\mathcal{O}_{\bar{S}}$. For $I = (i_1, \dots, i_{k}) \subset \{ 1,
\dots, d+1 \}^k $ we analyze the images of the elements $dx_I :=
dx_{i_1} \y \dots \y dx_{i_k}$, which generate    $\Omega_S^k$  as
a $\mathcal{O}_S$-algebra, by the homomorphism $\y^k \w$. We have
two possibilities.
\begin{enumerate}
 \item[(i)]
 If $i_1, \dots    i_k \in \{ 1, \dots, d \}$ then $\y^k \w ( \y^k \eta (dx_I )) =
 X^{\sum_I e_{i_r}} \otimes   e_{i_1} \y \cdots \y e_{i_k}$.
 Then taking wedges with elements of $\y^{d-k} M$ and applying
the homomorphism $\f$ we obtain a generator of $\J_k'$ equal to
$X^{e_{i_1} + \dots + e_{i_k}}$ if  and only if  $e_{i_1} \y
\cdots \y e_{i_k} \ne 0$.

\item[(ii)]  If one of the $i_r$, say $i_k$, is equal to $d+1$,
then we set $n:= \min\{ 1 \leqslant j \leqslant g \mid e_{i_1} \y
\cdots e_{i_{k-1}} \y \l_j \ne 0 \} \cup \{ \infty \}$. The image
of $x_{d+1}$ in $\mathcal{O}_{\bar{S}} = \C \{ \s^\vee \cap M \}$
is of the form $x_{d+1} = a + b$ where    $a =\sum _{\l \geqslant
\l_n} \b_\l X^{\l}$ and $b = \sum_ { \l \ngeq \l_n }    \b_\l
X^{\l}$. Then we have that $dx_{d+1} = da + db$ in
$\Omega_{\bar{S}}$ hence $\y^k
\w ( \y_{l=1, \dots, k-1} dx_{i_l} \y da) =0$ and
\[
\y^k \w  ( \y^k \eta ( dx_I ))  =  \y^k \w (\y_{l=1, \dots, k-1} dx_{i_l}
 \y db). \] Then taking wedges
with elements of $\y^{d-k} M$ and applying the homomorphism $\f$
we obtain a generator of $\J_k'$ of the form: \[ X^{e_{i_1} +
\dots + e_{i_{k-1}} + \l_n } \epsilon_n + \sum_{\mbox{{\small finite}}}
X^{e_{i_1} + \dots + e_{i_{k-1}} + v } \epsilon_{v}
\]
where $\epsilon_{\l_n}$ and $\epsilon_{v}$ are units in
$\mathcal{O}_{\bar{S}}$, $v \in \s^\vee \cap M_0$ and $e_{i_1} \y
\cdots \y e_{i_{k-1}}  \y v \ne 0$.
\end{enumerate}
The  elements we obtain in (i) and (ii) are generators of $\J_k'$.
Since $M_0 = \Z e_1 + \cdots + \Z e_d$  the terms
$X^{e_{i_1} + \dots + e_{i_{k-1}} + v }$ belong to $\J_k'$ hence
we deduce that $\J_k = \J_k'$. \end{proof}

\section{Examples}
\label{QOexample}

  \begin{Exam}        \label{e1}
 We compute the series    $P_{
{\geom}}^{(S,0)}(T)$ for a q.o.~surface parametrized by a q.o.~
branch $\z$ with characteristic exponents $\l_1=(3/2,0),\
\l_2=(7/4,0)\mbox{ and }\l_3=(2,1/2)$.
  \end{Exam}

We have that $\s=\R^2_{\geqslant 0}$ and $\s^\vee\cap M \cong
\Z^2_{\geqslant 0}$. It follows that $(S,0)$ has smooth
normalization thus $P_{\geom}^{(\bar{S},0)}(T) = (1 -
\L^2T)^{-1}$. We denote by $\theta_1$ and $\theta_2$ the
one-dimensional faces of $\s$. The plane curves $S_{\theta_1}$ and
$S_{\theta_2}$ have multiplicities two and four respectively which
determine the terms $P(S_{\theta_1})$ and $P(S_{\theta_2})$ by
Remark \ref{P-codim1QO}. Notice that $\mathcal D_1=\emptyset$ thus
\begin{equation} \label{aaa}
P(S)= \sum_{\t \in \Sigma_1\cap\Sigma_2}^{\stackrel{\circ}{\t}
\subset \stackrel{\circ}{\s}} P_{2, \t} (S) .
\end{equation} We
determine this sum by computing the generating functions of the
sets $\stackrel{\circ}{\t} \cap N$  and applying the method of
Proposition 9.5 and Section 12 of \cite{CoGP}. We get:
\begin{center}
$P_2(S)=\frac{(\L-1)^2}{1-\L^2T}\left(\frac{\L^{13}T^{17}}{(1-\L
T)(1-\L^{12}T^{16})}+\frac{\L^2T^6+\L^4T^8+\L^6T^{10}+\L^8T^{12}+\L^{10}T^{14}+\L^{12}T^{20}}{(1-\L^{12}T^{16})(1-T^4)}\right.$
  $\left.+\frac{\L^2T^4+\L^4T^8}{(1-T^4)(1-\L^4T^4)}+\frac{\L^{12}T^{16}}{1-\L^{12}T^{16}}+\frac{T^4}{1-T^4}\right).$
\end{center}
We have that  $      P_{
{\geom}}^{(S,0)}(T)=(1-T)^{-1}+P(S_{\theta_1})+P(S_{\theta_2})+P(S)$.
The motivic volume  is
\begin{center}
$\mu(H_S)=\frac{1}{(1-\L)(1-\L^{20})}+\frac{-1}{1-\L^{20}}+\frac{1+\L^6+\L^8+\L^{10}+\L^{12}+\L^{14}+\L^{16}+\L^{18}}{(1-\L^4)(1-\L^{20})}.$
\end{center}

 \begin{Exam} \label{e2}
 We describe the series    $P_{
{\geom}}^{(Z(S),0)}(T)$ associated with the monomial variety
associated with the q.o.~surface$(S,0)$  of example \ref{e1} (see
Definition \ref{zs}).
  \end{Exam}
 The semigroup $\Gamma$ is
determined by the characteristic exponents. In this example the
semigroup is generated by $e_1=(1,0),\ e_2=(0,1/2),\ e_3=(3/2,0),\
e_4=(13/4,0)$ and $e_5=(24/4,1/4)$. The lattice $M$ coincides with
the one of $\bar{S}$. The dual cone of $\Gamma \R_{\geqslant 0}$
coincides with $\s$. Since the monomial curves
$Z^{\Gamma\cap\theta_1^\perp}$ and $Z^{\Gamma \cap\theta_2^\perp}$
have multiplicities $2$ and $4$ respectively, we get $P(\Gamma
\cap\theta_i^\perp)= P(S_{\theta_i})$ for $i = 1, 2$. Since
$\mathcal D_1=\emptyset$ we obtain that $P(\Gamma)  $ is obtained
by a formula analogous to the right hand-side of (\ref{aaa}), in
terms of the subdivision associated to the logarithmic jacobian
ideals of $Z^{\Gamma}$ (see \cite{CoGP}). The subdivision obtained
for the monomial variety is different that the one obtained for
the q.o.~singularity. We get:
\begin{center}
$P(\Gamma)=\frac{(\L-1)^2}{1-\L^2T}\left(\frac{\L^{51}T^{55}}{(1-\L
T)(1-\L^{50}T^{54})}+\frac{\L^{50}T^{54}}{1-\L^{50}T^{54}}+\frac{\L^{50}T^{58}+\sum_{k=3}^{26}\L^{2k-4}T^{2k}}{(1-\L^{50}T^{54})(1-T^4)}+\right.$
$\left.\frac{T^4}{1-T^4}+\frac{\L^2T^4+\L^4T^8}{(1-T^4)(1-\L^4T^4)}\right).$
\end{center}
Then we have that $P_{
{\geom}}^{(Z(S),0)}(T)=(1-T)^{-1}+P(\Gamma\cap\theta_1^\perp)+P(\Gamma\cap\theta_2^\perp)+P(\Gamma)$,
and it is easy to see that $P_{\geom}^{(S,0)}(T)\neq
P_{\geom}^{(Z^\Gamma,0)}(T)$.

See \cite{Cobo} for explicit examples of geometric motivic
Poincar\'e series of toric or q.o.~singularities of dimensions
two or three.

\end{document}